\documentclass[nointlimits,reqno]{amsart}
\usepackage{amssymb,natbib,cases,enumerate,amsmath}
\usepackage{color}
\usepackage{hyperref}

\usepackage[a4paper,left=3cm,right=3cm,top=3cm,bottom=3cm,marginparsep=2pt] {geometry}

\theoremstyle{plain}
\newtheorem{theorem}{Theorem}[section]
\newtheorem{lemma}[theorem]{Lemma}

\theoremstyle{definition}
\newtheorem{remark}[theorem]{Remark}

\newtheorem{convention}[theorem]{Convention}
\numberwithin{equation}{section}

\newcommand{\Z}{\mathbb Z}
\newcommand{\K}{\mathcal K}
\newcommand{\M}{\mathfrak M}
\newcommand{\Mpl}{\mathfrak M_+}
\newcommand{\N}{\mathbb N}

\DeclareMathOperator*{\esssup}{ess\,sup}

\renewcommand{\d}{{\fam0 d}}

\newcommand{\Lpv}{\Lambda^{p}(v)}
\newcommand{\Gpuv}{\Gamma^{p}_u(v)}
\newcommand{\Gpuvweak}{\Gamma^{\infty}_u(v)}

\newcommand{\ds}{\,\mathrm d s}
\newcommand{\dt}{\,\mathrm d t}

\newcommand{\dy}{\,\mathrm d y}

\newcommand{\dnu}{\,\mathrm d \nu}
\newcommand{\dui}{\,\mathrm d U^{-1}}

\newcommand{\f}{f^*}

\makeatletter
\def\paragraph{\bigskip\@startsection{paragraph}{4}%
  \z@\z@{-\fontdimen2\font}%
  {\normalfont\bfseries}}
\makeatother

\begin{document}

\title[Discretization with no restrictions on weights]{Discretization and 
antidiscretization of Lorentz norms with~no~restrictions on weights}
\author{Martin K\v{r}epela, Zden\v ek Mihula, Hana Tur\v{c}inov\'{a}}

\address{Martin K\v{r}epela, Czech Technical University in Prague, Faculty of Electrical Engineering, Department of Mathematics, Technick\'a~2, 166~27 Praha~6, Czech Republic}
\email{martin.krepela@fel.cvut.cz}
\urladdr{\href{https://orcid.org/0000-0003-0234-1645}{0000-0003-0234-1645}}
\address{Zden\v ek Mihula, Czech Technical University in Prague, Faculty of Electrical Engineering, Department of Mathematics, Technick\'a~2, 166~27 Praha~6, Czech Republic --- and --- Charles University, Faculty of Mathematics and Physics, Department of Mathematical Analysis, Sokolovsk\'a~83, 186~75 Praha~8, Czech Republic}
\email{mihulzde@fel.cvut.cz}
\email{mihulaz@karlin.mff.cuni.cz}
\urladdr{\href{https://orcid.org/0000-0001-6962-7635}{0000-0001-6962-7635}}
\address{Hana Tur\v{c}inov\'{a}, Charles University, Faculty of Mathematics and Physics, Department of Mathematical Analysis, Sokolovsk\'a~83, 186~75 Praha~8, Czech Republic}
\email{turcinova@karlin.mff.cuni.cz}
\urladdr{\href{https://orcid.org/0000-0002-5424-9413}{0000-0002-5424-9413}}

\thanks{The first and second authors were supported by the project OPVVV CAAS CZ.02.1.01/0.0/0.0/16\_019/0000778. The second and third authors were supported by the grant P201-18-00580S of the Czech Science Foundation, by the grant SVV-2020-260583, and by Charles University Research program No.~UNCE/SCI/023.}

\begin{abstract}
        We improve the discretization technique for weighted Lorentz norms by eliminating all ``non-degeneracy'' restrictions on the involved weights.
        We use the new method to provide equivalent estimates on the optimal constant $C$ such that the inequality
                $$
                        \left( \int_0^L (\f(t))^{q} w(t)\dt \right)^\frac 
1{q} \le C \left( \int_0^L \left( \int_0^t u(s)\ds \right)^{-p} \left( \int_0^t \f(s) u(s) \ds \right)^p v(t) \dt \right)^\frac 1{p}
                $$
        holds for all relevant measurable functions, where $L\in(0,\infty]$, $p, q \in (0,\infty)$ and $u$, $v$, $w$ are locally integrable weights, $u$ being strictly positive.
        In the case of weights that would be otherwise excluded by the restrictions, it is shown that additional limit terms naturally appear in the characterizations of the optimal $C$.
        A~weak analogue for $p=\infty$ is also presented.
\end{abstract}

\date{\today}

\maketitle

\setcitestyle{numbers}
\bibliographystyle{plainnat}

\section{Introduction}

Consider the problem of determining the optimal (i.e., least) constant $C\in[0,\infty]$ such that the inequality
\begin{equation}\label{I:main}
        \left( \int_0^L (\f(t))^{q} w(t)\dt \right)^\frac 1{q} \le C \left( \int_0^L \left( \frac 1{U(t)} \int_0^t \f(s) u(s) \ds \right)^p v(t) \dt \right)^\frac 1{p}
\end{equation}
is satisfied for all functions $f$ defined on an~``appropriate'' measure space, where $\f$ stands for the nonincreasing rearrangement of $f$ and $U(t)=\int_0^t u(s)\ds$. The values of the other involved parameters are 
fixed, namely $L\in(0,\infty]$, $p, q \in (0,\infty)$, and $u$, $v$, $w$ are locally integrable, nonnegative weights on $(0,L)$, $u$ strictly positive.

In other words, this corresponds to the problem of characterizing the embedding $\Gamma^{p}_u(v) \hookrightarrow \Lambda^{q}(w)$ (all the involved 
notation is discussed in Section~\ref{sec:prel} below).
This problem has been extensively studied and there are several possible approaches leading to a~solution. Let us briefly inspect what is at our disposal.

Gogatishvili and Pick provided in \cite{GP:03} what is currently the most 
cited solution to the problem. It relies on a~method of discretization, in which the integral expressions in \eqref{I:main} are reformulated in terms of specific sequences. What was innovative in their paper was the so-called ``antidiscretization'' part, where the discrete conditions were transformed back into integral ones. The technique from \cite{GP:03} is our 
point of departure and will be discussed in detail. We note that, in what 
follows, by ``discretization'' we will actually refer to the whole process including the antidiscretization part.

Although \cite{GP:03} satisfied the demand for conditions in a~form that may be easily verified, there is a~catch. Namely, only the case $L=\infty$ is covered and it is assumed there that $v$ is ``non-degenerate'' with respect to $u$ in the sense that
\begin{equation*}
        \int_0^\infty \frac{v(s)\ds}{U^p (s) + U^p (t)} < \infty \quad\text{ for all } t\in(0,\infty),\qquad \int_0^1 \frac{v(s)\ds}{U^p (s)} = \int_1^\infty v(s) \ds = \infty.
\end{equation*}
It turns out that the first of these conditions can be assumed without loss of generality (see the beginning of the proof of Theorem~\ref{thm:HanickaToLambda}); therefore, what rules out ``degenerated'' weights is the condition
\begin{equation}\label{intro:condofnondegenerancy}
\int_0^1 \frac{v(s)\ds}{U^p (s)} = \int_1^\infty v(s) \ds = \infty.
\end{equation}
Unfortunately, this means that, besides leaving out some ``degenerated'' weights on $(0,\infty)$, the result cannot be used in any direct way (e.g., using the obvious idea of truncating $v$ and $w$) in the case where $L<\infty$. A~finite $L$ appears naturally when the considered weighted Lorentz spaces consist of functions defined on a finite measure space. Such a setting is perfectly reasonable and it even becomes inevitable when embeddings of weighted Sobolev-Lorentz spaces on domains into Lorentz $\Lambda$-spaces and/or their compactness are studied, which is an~application that we have in mind. More details on this matter are given in Remark~\ref{rem:sob}.

A~completely different way of approaching the problem \eqref{I:main} was found independently by Sinnamon in \cite{Sinnamon02, Sinnamon2003}. It is 
based on reformulating \eqref{I:main} as an~inequality for quasinconcave functions. His method is actually far simpler than discretization. However, the goal of \cite{Sinnamon02} was to describe the K\"othe dual of Lorentz $\Gamma$-spaces; therefore only the case $q = 1$ and $u\equiv 1$ was considered there. It appears that there is no easy way to modify his proof technique to allow other values of $q$. Nevertheless, the result obtained in \cite{Sinnamon02} does not require any additional assumptions 
on weights and it gives a hint on how the conditions characterizing \eqref{I:main} change in the general case. Namely, there appear certain limit terms of the same nature as in other embeddings between Lorentz $\Lambda$ 
and $\Gamma$-spaces (cf.~\cite{CPSS}). Another description of the K\"othe 
dual of Lorentz $\Gamma$-spaces was given by Gogatishvili and Kerman in \cite{GK}. Their method is not built on discretization either, and so their result does not require any extra assumptions on weights, but, again, it covers the problem \eqref{I:main} only for $q = 1$ and $u\equiv 1$.

In \cite{EGO:18}, the discretization technique was modified in order to encompass ``degenerated'' weights as well. Our goal in the present paper is to use this modification to finally provide a~complete characterization 
of the optimal constant $C$ in \eqref{I:main} without the restriction \eqref{intro:condofnondegenerancy} on the 
weights $u$ and $v$, and for all positive values of $p$ and $q$, 
including the ``weak-type'' modification for $p=\infty$ as well.

It should be noted that inequality \eqref{I:main} with a~general $u$ in fact follows from the case $u\equiv 1$. Indeed, since $u$ is locally integrable and positive a.e.~in $(0,L)$, the function $U$ is absolutely continuous and $U'>0$ a.e.~in $(0,L)$; hence its inverse $U^{-1}$ is also absolutely continuous. Thus, performing the change of variables $t\mapsto U^{-1}(t)$ and considering that, since $U^{-1}$ is strictly increasing, a~function $h$ is nonincreasing if and 
only if $h\circ U^{-1}$ is nonincreasing, we observe that \eqref{I:main} holds for all $f\in\M(0,L)$ if and only if 
	$$
		\left( \int_0^{U(L)} (g^*(t))^{q} w(U^{-1}(t))\dui(t) \right)^\frac 1{q} \le C \left( \int_0^{U(L)} \left( \frac 1t \int_0^t g^*(s) \ds \right)^p v(U^{-1}(t)) \dui(t) \right)^\frac 1{p}
	$$
holds for all $g\in\M(0,U(L))$ (in here, $(0,L)$ may be, of course, replaced by any nonatomic measure space of measure $L$). Hence, to tackle \eqref{I:main} it suffices to consider $u\equiv 1$ and, at the end, perform the suggested change of variables to get the general version. Nevertheless, in the proofs of the main results in Section 4, we use the discretization technique directly with the general $u$. There would be little difference if we used $u\equiv1$, namely only in writing $t$ instead of $U(t)$ in the proofs. By using $U(t)$ we also avoid the need for performing 
the substitution to obtain the final result.

\section{Preliminaries}\label{sec:prel}

Let us summarize the notation and auxiliary results that we shall use in this paper.
Throughout the paper, $L\in(0,\infty]$ is a fixed positive (possibly infinite) number.
\begin{convention}\label{conv}  We adhere to the following conventions:
        \begin{enumerate}[(i)]
                \item\label{conv:limit} If $f$ is a function on $(0,L)$,  
then $f(0)$ and $f(L)$ stand for $\lim_{t\to0^+}f(t)$ and $\lim_{t\to L^-}f(t)$, respectively.
                \item\label{conv:indefinite} All of the expressions $\frac1{\infty}$, $\frac{\infty}{\infty}$, $\frac{0}{0}$ and $0\cdot\infty$ are to be interpreted as $0$. The expression $\frac{1}{0}$ is to be interpreted as $\infty$.
        \end{enumerate}
\end{convention}

Let $(X,\mu)$ be a~nonatomic, $\sigma$-finite measure space such that $\mu(X)=L$. By $\M_\mu (X)$ we denote the set of all $\mu$-measurable (extended) real-valued functions defined on $X$. The symbol $\M(0,L)$ denotes 
the set of all Lebesgue-measurable functions on $(0,L)$, and $\Mpl(0,L)$ denotes the set of all $f\in\M(0,L)$ such that $f\ge 0$ a.e.\ on $(0,L)$.

We say that a function $v\in\Mpl(0,L)$ is a \emph{weight} on $(0,L)$ if $0<V(t)<\infty$ for every $t\in(0,L)$, where
\begin{equation*}
        V(t)=\int_0^tv(s) \ds, \quad t\in[0,L].
\end{equation*}
Furthermore, we denote
\begin{equation*}
        V(a,b)=\int_a^bv(s) \ds,\quad 0\le a < b\le L.
\end{equation*}

If $f\in\M_\mu (X)$, the symbol $\f$ denotes the \emph{nonincreasing rearrangement} of $f$, that is,
\begin{equation*}
f^*(t)=\inf\{\lambda\in[0,\infty)\colon\mu(\{x\in X\colon|f(x)|>\lambda\})\leq t\},\ t\in(0, L),
\end{equation*}
(for details see \cite{BS}).
Let $p\in(0,\infty]$ and $v$ be a weight on $(0,L)$.
We define the following functionals: 
\begin{equation*}
        \|f\|_{\Lpv}  = \begin{cases}
                \displaystyle\left( \int_0^L (\f(t))^pv(t)\dt \right)^\frac1p\quad&\text{if } p\in(0,\infty),\\
                \displaystyle\esssup_{t\in(0,L)}f^*(t)v(t)\quad&\text{if } p=\infty.
        \end{cases}
\end{equation*}
Let $u$ be an a.e.~positive weight on $(0, L)$. Let
\begin{equation*}
        f_u^{**}(t)=\frac1{U(t)}\int_0^t\f(s)u(s)\ds,\ t\in(0,L),
\end{equation*}
be the nonincreasing maximal function of $f$ with respect to $u$ (cf.~\cite{GP:03}).
We define the functional
\begin{equation*}
        \|f\|_{\Gpuv} = \|f^{**}_u\|_{\Lpv}.
\end{equation*}
Accordingly, we denote
	\begin{align*}
	        \Lpv & = \left\{ f\in\M_\mu (X):\ \|f\|_{\Lpv}<\infty \right\},\\
	        \Gamma_u^p(v) & = \left\{ f\in\M_\mu (X):\ \|f\|_{\Gpuv}<\infty \right\}.
	\end{align*}
These classes of functions are the usual weighted Lorentz $\Lambda$ and $\Gamma$-spaces (cf.~\cite{CPSS,GP:03}).\\

A~function $\varrho\colon(0,L)\to(0,\infty)$ is called \emph{admissible} if it is positive, increasing and continuous on $(0,L)$.
If $\varrho$ is admissible, we say that a function $h\colon(0,L)\to[0,\infty)$ is \emph{$\varrho$-quasiconcave}, and we write $h\in Q_\varrho(0,L)$, if $h$ is nondecreasing on $(0,L)$ and the function $\frac{h}{\varrho}$ is nonincreasing on $(0,L)$. Thanks to the monotonicity properties of $\varrho$-quasiconcave functions, it follows that $h\not\equiv0$ on $(0,L)$ if and only if $h(t)\neq0$ for each $t\in(0,L)$. Throughout the paper, we implicitly assume that $\varrho$ is an admissible function on $(0,L)$.

If $h\not\equiv0$ is a $\varrho$-quasiconcave function, so is the function $\frac{\varrho}{h}$. Furthermore, the function $h^p$ is $\varrho^p$-quasiconcave for each $p>0$. A~nonnegative linear combination of $\varrho$-quasiconcave functions is also $\varrho$-quasiconcave. If functions $h_1$ and $h_2$ are $\varrho_1$-quasiconcave and $\varrho_2$-quasiconcave, respectively, then $h_1\cdot h_2$ is $(\varrho_1\cdot\varrho_2)$-quasiconcave.

Every $h\in Q_\varrho(0,L)$ has an~integral representation with limit terms (see~\citep[Theorem~2.4.1]{EGO:18}). Precisely, there is a nonnegative 
Borel measure $\nu$ on $(0,L)$ such that
\begin{equation}\label{prel:repreofQq}
        h(t)\le \lim_{s\to0^+}h(s) + \left(\lim_{s\to L^-}\frac{h(s)}{\varrho(s)}\right)\varrho(t)+\int_{(0,L)}\min\{\varrho(t), \varrho(s)\}\dnu(s)\le4 h(t)\quad\text{for each $t\in(0,L)$}.
\end{equation}
For more information on $\varrho$-quasiconcave functions, see~\citep[Chapter~2]{EGO:18}.\\

The cornerstone of the discretization technique is the construction of a~covering sequence. The properties of such a~sequence, as listed below, as 
well as their proofs can be found in \citep[Chapter~3]{EGO:18}. For every 
$h\in Q_\varrho(0,L)$, $h\not\equiv0$, and each $a>1$, there are numbers $K_-,K^+\in\{\Z,\pm\infty\}$ with $-\infty\le K_-\le0\le K^+\le\infty$, and a sequence $\{x_k\}_{k\in\K_-^+}$, where $\K^+_-=\{k\in\Z\colon K_-\le k\le K^+\}$, with the following properties:
\begin{itemize}
        \item The sequence $\{x_k\}_{k\in\K_-^+}$ is increasing and $x_k\in (0,L)$ for every $k\in\Z$ such that $K_-+1 \le k \le K^+-1$.
       
        \item $K^+=\infty$ if and only if
        \begin{equation}\label{prel:CSdegenL}
                \lim_{t\to L^-}h(t)=\infty\quad\text{and}\quad\lim_{t\to L^-}\frac{\varrho(t)}{h(t)}=\infty.
        \end{equation}
        If $K^+=\infty$, then $\lim_{k\to\infty}x_k=L$. Otherwise, $x_{K^+}=L$.
       
        \item $K_-=-\infty$ if and only if
        \begin{equation}\label{prel:CSdegen0}
                \lim_{t\to 0^+}h(t)=0\quad\text{and}\quad\lim_{t\to 0^+}\frac{\varrho(t)}{h(t)}=0.
        \end{equation}
        If $K_-=-\infty$, then $\lim_{k\to-\infty}x_k=0$. Otherwise, $x_{K_-}=0$.
       
        \item For every $k\in\Z$ such that $K_-+2 \le k \le K^+-1$, one has
        \begin{equation}\label{prel:CSgeomincr}
                a h(x_{k-1})\le h(x_k)\quad\text{and}\quad a\frac{\varrho(x_{k-1})}{h(x_{k-1})}\le\frac{\varrho(x_k)}{h(x_k)}.
        \end{equation}
       
        \item For every $k\in\Z$ such that $K_-+2 \le k \le K^+-1$, one has
        \begin{align*}
                \frac1{a}h(x_k)\le h(t)\le h(x_k)\quad&\text{for each $t\in[x_{k-1},x_k]$}\\
                \intertext{or}
                \frac1{a}\frac{\varrho(x_k)}{h(x_k)}\le\frac{\varrho(t)}{h(t)}\le\frac{\varrho(x_k)}{h(x_k)}\quad&\text{for each $t\in[x_{k-1},x_k]$}.
        \end{align*}
       
        \item If $K^+<\infty$, then
        \begin{align*}
                h(x_{K^+-1})\le h(t)\le a h(x_{K^+-1})\quad&\text{for each $t\in[x_{K^+-1},L)$}\\
                \intertext{or}
                \frac{\varrho(x_{K^+-1})}{h(x_{K^+-1})}\le \frac{\varrho(t)}{h(t)}\le a \frac{\varrho(x_{K^+-1})}{h(x_{K^+-1})}\quad&\text{for each $t\in[x_{K^+-1},L)$}.
        \end{align*}
       
        \item If $K_->-\infty$, then
        \begin{align*}
                \frac1{a}h(x_{K_-+1})\le h(t)\le h(x_{K_-+1})\quad&\text{for each $t\in(0, x_{K_-+1}]$}\\
                \intertext{or}
                \frac1{a}\frac{\varrho(x_{K_-+1})}{h(x_{K_-+1})}\le \frac{\varrho(t)}{h(t)}\le\frac{\varrho(x_{K_-+1})}{h(x_{K_-+1})}\quad&\text{for each $t\in(0, x_{K_-+1}]$}.
        \end{align*}
\end{itemize}
If $\{x_k\}_{k\in\K_-^+}$ satisfies these conditions, it is called a \emph{covering sequence (with respect to $h,$ $\varrho$, $a$).}

The family of all covering sequences with respect to $h,$ $\varrho$, $a$ is denoted by $CS(h,\varrho,a)$ (we omit any reference to the interval $(0,L)$ in this notation because it will always be apparent what the underlying interval is). We also denote
\begin{align*}
        \K          & = \{k\in\Z\colon K_-+1 \le k \le K^+-1\} \\
        \K^+ & = \{k\in\Z\colon K_-+1 \le k \le K^+\}.
\end{align*}
The properties of $\{x_k\}_{k\in\K_-^+}$ imply that
\begin{equation}\label{prel:CSunion}
        (0,L)\subseteq\bigcup_{k\in\K^+}(x_{k-1},x_k]\subseteq(0,L],
\end{equation}
where the first inclusion is strict if and only if $K^+\ne \infty$.

Furthermore, if $\{x_k\}_{k\in\K_-^+}\in CS(h,\varrho,a)$, then there is a decomposition
\begin{equation}\label{prel:CSdecomp}
        \K^+=\mathcal{Z}_1\cup\mathcal{Z}_2\quad\text{and}\quad\mathcal{Z}_1\cap\mathcal{Z}_2=\emptyset
\end{equation}
such that
\begin{align}
        h(t)\approx h(x_k)\quad&\text{for every $t\in[x_{k-1},x_k]$ and each $k\in\mathcal{Z}_1$}, \label{prel:CSZ1}\\
        \intertext{and}
        \frac{\varrho(t)}{h(t)}\approx \frac{\varrho(x_k)}{h(x_k)}\quad&\text{for every $t\in[x_{k-1},x_k]$ and each $k\in\mathcal{Z}_2$}\label{prel:CSZ2},
\end{align}
where the equivalence constants depend only on $a$.
Moreover, if $\{x_k\}_{k\in\K_-^+}\in CS(h,\varrho,a)$, then $\{x_k\}_{k\in\K_-^+}\in CS(\frac{\varrho}{h},\varrho,a)$, and $\{x_k\}_{k\in\K_-^+}\in CS(h^p,\varrho^p,a^p)$ for every $p\in(0,\infty)$.

\section{Discretization of generalized Lorentz norms}

This section contains technical lemmas necessary for implementing discretization to solve our main problem. The results below extend their counterparts in \cite{GP:03} by eliminating the ``non-degeneracy'' assumptions in there. Namely, Lemmas \ref{lem:disantidissup}, \ref{lem:disantidisint} and \ref{lem:discantisup} below correspond, in this sense, to Lemmas 3.7, 
3.6 and 3.8 in \cite{GP:03}, respectively.\\

We start with an~auxiliary result that is frequently used when one deals with weighted inequalities.

\begin{lemma}\label{lemma:changeofvariable}
	Let $L\in(0,\infty]$ and let $v$ be a weight on $(0,L)$. Let $0\le a<b\le L$. If $\gamma>-1$, then
	\begin{equation*}
		V^{\gamma+1}(a,b)=(\gamma+1)\int_a^b{\left(\int_a^tv(s)\ds \right)^\gamma v(t) \dt}.
	\end{equation*}
\end{lemma}
\begin{proof}
	Assume that $b<L$. Since $v$ is a weight on $(0,L)$, the function $\psi(t)=V(a,t)$, $t\in[a,b]$, is absolutely continuous on $[a,b]$. Hence the 
claim follows from the change of variables $y=\psi$ in the integral on the right-hand side (e.g., \citep[page~156]{MR924157}).
	
	If $b=L$, the claim follows from the part already proved and the monotone convergence theorem.
\end{proof}

The following theorem generalizes \citep[Corollary~2.13]{GP:03} by allowing 
degenerated weights (cf.~\citep[Lemma~4.1.1]{EGO:18}).
\begin{theorem}\label{thm:disantidis}
	Let $p\in(0,\infty)$ and $h\in Q_{\varrho}(0,L)$. Assume that there exist $C_1,\ C_2\in (0,\infty)$, $\alpha,\beta\in[0,\infty)$ and a nonnegative Borel measure $\nu$ on $(0,L)$ such that
	\begin{equation}\label{thm:disantidis:repre}
		C_1h(t)\le\alpha+\beta \varrho(t)+\int_{(0,L)}\min\{\varrho(t), \varrho(s)\}\dnu(s)\le C_2h(t)\quad\text{for every } t\in(0,L).
	\end{equation}
	If $a>0$ is sufficiently large, $\{x_k\}_{k\in\K^+_-}\in CS(h,\varrho,a)$ and $f\in Q_{\varrho^p}(0,L)$, then
	\begin{align}
		\sum_{k\in\K^+_-}\frac{f(x_k)h^p(x_k)}{\varrho^p(x_k)}\ &\ \approx
		\alpha^p\lim_{t\to0^+}\frac{f(t)}{\varrho^p(t)}+\beta^p\lim_{t\to L^-}f(t) \notag \\
		& \qquad+\int_{(0,L)}f(t)\varrho^{1-p}(t) \left( \int_{(0,L)}\min\{\varrho(t),\varrho(s)\} \dnu(s) \right)^{p-1} \dnu(t). \label{thm:disantidis:eq}
	\end{align}
	Precisely, it is sufficient if the parameter $a$ satisfies
	\begin{equation*}
		a^p>12\cdot\frac{3^{p+\max\{1, p\}}C^p_2}{\min\{1,p\}C^p_1}.
	\end{equation*}
	Moreover, the equivalence constants in \eqref{thm:disantidis:eq} depend only on $p, a, C_1, C_2$.
\end{theorem}

\begin{proof}
	Without loss of generality, we may assume that $h\not\equiv0$ and $f\not\equiv0$. Since $h$ and $f$ are in $Q_{\varrho}(0,L)$ and $Q_{\varrho^p}(0,L)$, respectively, we have $h\neq0$ and $f\neq0$ on $(0,L)$.
	
	It can be easily shown that the limit terms on the left-hand side of \eqref{thm:disantidis:eq} are well defined if they are to appear (recall Convention~\ref{conv}(\ref{conv:limit})). This follows from the fact that $h\in Q_{\varrho}(0,L)$ and from the properties of $\{x_k\}_{k\in\K_-^+}$ (cf.~\citep[the proof of Lemma~4.1.1]{EGO:18}). Furthermore, the limit terms on the right-hand side of \eqref{thm:disantidis:eq} are always well defined thanks to $f\in Q_{\varrho^p}(0,L)$ and Convention~\ref{conv}(\ref{conv:indefinite}).
	
	In order to simplify the notation, we set
	\begin{equation*}
		\psi(t) = \varrho^{1-p}(t) \left( \int_{(0,L)}\min\{\varrho(t),\varrho(s)\} \dnu(s) \right)^{p-1},\ t\in(0,L).
	\end{equation*}
	
	We begin the proof of \eqref{thm:disantidis:eq} by showing that
	\begin{equation}\label{thm:disantidis:eq1}
		\sum_{k\in\K^+_-}\frac{f(x_k)h^p(x_k)}{\varrho^p(x_k)}\gtrsim\alpha^p\lim_{t\to0^+}\frac{f(t)}{\varrho^p(t)}+\beta^p\lim_{t\to L^-}f(t)+\int_{(0,L)}f(t)\psi(t) \dnu(t).
	\end{equation}
	We first check that
	\begin{equation*}
		\alpha^p\lim_{t\to0^+}\frac{f(t)}{\varrho^p(t)}+\beta^p\lim_{t\to L^-}f(t)\lesssim\sum_{k\in\K^+_-}\frac{f(x_k)h^p(x_k)}{\varrho^p(x_k)}.
	\end{equation*}
	If $\alpha=\beta=0$, the inequality holds trivially. Assume that $\alpha>0$.
	We have $\lim_{t\to0^+}h(t)\gtrsim\alpha>0$ thanks to \eqref{thm:disantidis:repre}, and thus $K_->-\infty$ by \eqref{prel:CSdegen0}. Furthermore, 
we have
	\begin{align*}
		\alpha^p\lim_{t\to0^+}\frac{f(t)}{\varrho^p(t)}&\lesssim\left(\lim_{t\to0^+}h^p(t)\right)\left(\lim_{t\to0^+}\frac{f(t)}{\varrho^p(t)}\right)=\lim_{t\to0^+}\frac{f(t)h^p(t)}{\varrho^p(t)}=\frac{fh^p}{\varrho^p}(x_{K_-})\\
		&\leq\sum_{k\in\K^+_-}\frac{f(x_k)h^p(x_k)}{\varrho^p(x_k)}.
	\end{align*}
	Note that the first equality is indeed valid because both limits are positive. One can similarly prove that $\beta^p\lim_{t\to L^-}f(t)\lesssim\sum_{k\in\K^+_-}\frac{f(x_k)h^p(x_k)}{\varrho^p(x_k)}$ owing to \eqref{prel:CSdegenL}.
	
	Hence, in order to prove \eqref{thm:disantidis:eq1} it remains to show
	\begin{equation}\label{thm:disantidis:eq2}
		\int_{(0,L)}f(t)\psi(t)\dnu(t)\lesssim\sum_{k\in\K^+_-}\frac{f(x_k)h^p(x_k)}{\varrho^p(x_k)}.
	\end{equation}
	We have
	\begin{equation}\label{thm:disantidis:powofhrep}
		\widetilde{C}_1h^p(t)\leq\alpha^p+\beta^p\varrho^p(t)+\int_{(0,L)}\min\{\varrho^p(t), \varrho^p(s)\}\psi(s)\dnu(s)\le\widetilde{C}_2h^p(t)\quad\text{for every $t\in(0,L)$},
	\end{equation}
	where $\widetilde{C}_1=\frac{C_1^p}{3^{p+\max\{p,1\}}}$ and $\widetilde{C}_2=\frac{6C_2^p}{\min\{1,p\}}$, owing to \citep[Theorem~2.4.3 and its proof]{EGO:18}.
	Let $\K^+=\mathcal{Z}_1\cup\mathcal{Z}_2$ be a decomposition of $\K^+$ 
from \eqref{prel:CSdecomp}. Denote by $\widetilde{\nu}$ the extension of $\nu$ to $(0,L]$ by zero. Using \eqref{prel:CSZ1},  the fact that $f\in Q_{\varrho^p}(0,L)$, and \eqref{thm:disantidis:powofhrep}, one has
	\begin{align*}
		\sum_{k\in\mathcal{Z}_1}\int_{(x_{k-1},x_k]}f(t)\psi(t)\,\d\widetilde{\nu}(t)&=\sum_{k\in\mathcal{Z}_1}\int_{(x_{k-1},x_k]}\frac{f(t)}{\varrho^p(t)}\varrho^p(t)\psi(t)\,\d\widetilde{\nu}(t)\\
		&\le\sum_{k\in\mathcal{Z}_1}\frac{f(x_{k-1})}{\varrho^p(x_{k-1})}\int_{(x_{k-1},x_k]}\min\{\varrho^p(x_k), \varrho^p(t)\}\psi(t)\,\d\widetilde{\nu}(t)\\
		&\lesssim\sum_{k\in\mathcal{Z}_1}\frac{f(x_{k-1})}{\varrho^p(x_{k-1})}h^p(x_k)\approx\sum_{k\in\mathcal{Z}_1}\frac{f(x_{k-1})}{\varrho^p(x_{k-1})}h^p(x_{k-1})\\
		&\le\sum_{k\in\K^+_-}\frac{f(x_k)h^p(x_k)}{\varrho^p(x_k)}.
	\end{align*}
	Here and below, if $x_{k-1}=0$ or $x_{k}=L$, the corresponding terms 
are to be understood as the corresponding limits. Hence
	\begin{equation}\label{thm:disantidis:eq3}
		\sum_{k\in\mathcal{Z}_1}\int_{(x_{k-1},x_k]}f(t)\psi(t)\,\d\widetilde{\nu}(t)\lesssim\sum_{k\in\K^+_-}\frac{f(x_k)h^p(x_k)}{\varrho^p(x_k)}.
	\end{equation}
	Furthermore, by \eqref{prel:CSZ2}, by the fact that $h^p\in Q_{\varrho^p}(0,L)$, by the fact that $f$ is nondecreasing, and by \eqref{thm:disantidis:powofhrep},
	\begin{align*}
		\sum_{k\in\mathcal{Z}_2}\int_{(x_{k-1},x_k]}f(t)\psi(t)\,\d\widetilde{\nu}(t)&\le\sum_{k\in\mathcal{Z}_2}f(x_k)\frac{\varrho^p(x_{k-1})}{\varrho^p(x_{k-1})}\int_{(x_{k-1},x_k]}\psi(t)\,\d\widetilde{\nu}(t)\\
		&=\sum_{k\in\mathcal{Z}_2}f(x_k)\frac1{\varrho^p(x_{k-1})}\int_{(x_{k-1},x_k]}\min\{\varrho^p(x_{k-1}),\varrho^p(t)\}\psi(t)\,\d\widetilde{\nu}(t)\\
		&\lesssim\sum_{k\in\mathcal{Z}_2}f(x_k)\frac{h^p(x_{k-1})}{\varrho^p(x_{k-1})}\approx\sum_{k\in\mathcal{Z}_2}f(x_k)\frac{h^p(x_k)}{\varrho^p(x_k)}\\
		&\le\sum_{k\in\K^+_-}\frac{f(x_k)h^p(x_k)}{\varrho^p(x_k)}.
	\end{align*}
	
	Hence
	\begin{equation}\label{thm:disantidis:eq4}
		\sum_{k\in\mathcal{Z}_2}\int_{(x_{k-1},x_k]}f(t)\psi(t)\,\d\widetilde{\nu}(t)\lesssim\sum_{k\in\K^+_-}\frac{f(x_k)h^p(x_k)}{\varrho^p(x_k)}.
	\end{equation}
	Desired inequality \eqref{thm:disantidis:eq2} now follows from \eqref{thm:disantidis:eq3}, \eqref{thm:disantidis:eq4} and \eqref{prel:CSunion}.

	Our next goal is to prove the estimate
	\begin{equation}\label{thm:disantidis:eq5}
		\sum_{k\in\K^+_-}\frac{f(x_k)h^p(x_k)}{\varrho^p(x_k)}\lesssim\alpha^p\lim_{t\to0^+}\frac{f(t)}{\varrho^p(t)}+\beta^p\lim_{t\to L^-}f(t)+\int_{(0,L)}f(t)\psi(t) \dnu(t).
	\end{equation}
	We start by showing that
	\begin{equation}\label{thm:disantidis:eq6}
		\int_{(0,L)}f(t)\psi(t)\dnu(t)\gtrsim\sum_{k=K_-+2}^{K^+-2}\frac{f(x_k)h^p(x_k)}{\varrho^p(x_k)}.
	\end{equation}
	By \eqref{prel:CSgeomincr}, we have
	\begin{align*}
		h^p(x_{k-1})\le\frac1{a^p}h^p(x_k)\quad&\text{for every $k\in\Z$, $K_-+2\le k\le K^+-1$}\\
		\intertext{and}
		\frac{\varrho(x_k)^p}{\varrho(x_{k+1})^p}h^p(x_{k+1})\le\frac1{a^p}h^p(x_k)\quad&\text{for every $k\in\Z$, $K_-+1\le k\le K^+-2$}.
	\end{align*}
	By combining these two inequalities we obtain
	\begin{equation}\label{thm:disantidis:eq7}
		\widetilde{C}_1h^p(x_k)-\widetilde{C}_2h^p(x_{k-1})-\widetilde{C}_2\frac{\varrho^p(x_k)}{\varrho^p(x_{k+1})}h^p(x_{k+1})\ge\left(\widetilde{C}_1-\frac{2\widetilde{C}_2}{a^p}\right)h^p(x_k)
	\end{equation}
	for every $k\in\Z$, $K_-+2\le k\le K^+-2$.
	Since $f\in Q_{\varrho^p}(0,L)$ and \eqref{prel:CSunion}, we have that
	\begin{equation}\label{thm:disantidis:eq8}
		\begin{aligned}
			2\int_{(0,L)}f(t)\psi(t)\dnu(t)&\ge\sum_{k=K_-+2}^{K^+-2}\int_{(x_{k-1},x_{k+1}]}f(t)\psi(t)\dnu(t)\\
			&\ge\sum_{k=K_-+2}^{K^+-2}\Bigg(\frac{f(x_k)}{\varrho^p(x_k)}\int_{(x_{k-1},x_{k}]}\varrho^p(t)\psi(t)\dnu(t)\\
			&\quad+ \frac{f(x_k)}{\varrho^p(x_k)}\int_{(x_k,x_{k+1}]}\varrho^p(x_k)\psi(t)\dnu(t)\Bigg)\\
			&=\sum_{k=K_-+2}^{K^+-2}\frac{f(x_k)}{\varrho^p(x_k)}\left(\int_{(x_{k-1},x_{k+1}]}\min\{\varrho^p(x_k),\varrho^p(t)\}\psi(t)\dnu(t)\right).
		\end{aligned}
	\end{equation}
	Using \eqref{thm:disantidis:powofhrep} and \eqref{thm:disantidis:eq7}, we have
	\begin{align*}
		& \int_{(x_{k-1},x_{k+1}]}\min\{\varrho^p(x_k),\varrho^p(t)\}\psi(t)\dnu(t)\\
		& \quad = \alpha^p + \beta^p \varrho^p(x_k) + \int_{(0,L)}\min\{\varrho^p(x_k), \varrho^p(t)\}\psi(t)\dnu(t)
		-\left(\alpha^p + \int_{(0,x_{k-1}]}\min\{\varrho^p(x_{k-1}), \varrho^p(t)\}\psi(t)\dnu(t)\right)\\
		& \quad \qquad- \frac{\varrho^p(x_k)}{\varrho^p(x_{k+1})}\left(\beta^p \varrho^p(x_{k+1}) + \int_{(x_{k+1}, L)}\min\{\varrho^p(x_{k+1}), \varrho^p(t)\}\psi(t)\dnu(t)\right)\\
		& \quad \ge\widetilde{C}_1h^p(x_k) - \widetilde{C}_2h^p(x_{k-1})-\frac{\varrho^p(x_k)}{\varrho^p(x_{k+1})}\widetilde{C}_2h^p(x_{k+1})\\
		& \quad \ge\left(\widetilde{C}_1-\frac{2\widetilde{C}_2}{a^p}\right)h^p(x_k)
	\end{align*}
	for each $k\in\Z$, $K_-+2\le k\le K^+-2$. Note that choosing a sufficiently large parameter $a$ assures that $\widetilde{C}_1-\frac{2\widetilde{C}_2}{a^p}>0$. Hence inequality \eqref{thm:disantidis:eq6} follows from the last chain of inequalities and \eqref{thm:disantidis:eq8}.
	
	If $K_-=-\infty$ and $K^+=\infty$ (and so $K_-+2=K_-$ and $K^+-2=K^+$), inequality \eqref{thm:disantidis:eq6} clearly implies inequality \eqref{thm:disantidis:eq5}. Thus, the proof is finished when $K_-=-\infty$ and $K^+=\infty$.
	Now suppose that $K_->-\infty$ or $K^+<\infty.$ Since $f\in Q_{\varrho^p}(0,L)$, we have
	\begin{align*}
		\int_{(0,L)}f(t)\psi(t)\dnu(t)
		&= \int_{(0,x]}f(t)\psi(t)\dnu(t) + \int_{(x,L)}f(t)\psi(t)\dnu(t)\\
		&\ge \frac{f(x)}{\varrho^p(x)}\int_{(0,x]}\min\{\varrho^p(x),\varrho^p(t)\}\psi(t)\dnu(t)
		+ \frac{f(x)}{\varrho^p(x)}\int_{(x,L)}\min\{\varrho^p(x), \varrho^p(t)\}\psi(t)\dnu(t)\\
		&=\frac{f(x)}{\varrho^p(x)}\int_{(0,L)}\min\{\varrho^p(x), \varrho^p(t)\}\psi(t)\dnu(t)
	\end{align*}
	for every $x\in(0,L)$, whence
	\begin{equation}\label{thm:disantidis:eq9}
		\begin{aligned}
			&\alpha^p\lim_{t\to0^+}\frac{f(t)}{\varrho^p(t)}+\beta^p\lim_{t\to L^-}f(t)+\int_{(0,L)}f(t)\psi(t)\dnu(t)\\
			&\ge\frac{f(x)}{\varrho^p(x)}\left(\alpha^p+\beta^p\varrho^p(x)+\int_{(0,L)}\min\{\varrho^p(x), \varrho^p(t)\}\psi(t)\dnu(t)\right)\\
			&\approx\frac{f(x)h^p(x)}{\varrho^p(x)}.
		\end{aligned}
	\end{equation}
	Assume, for example, $K_->-\infty$ and $K^+=\infty$ (and so $K^+-2=K^+$ and $\beta=0$), \eqref{thm:disantidis:eq9} together with \eqref{thm:disantidis:eq6} implies that
	\begin{align*}
		\alpha^p\lim_{t\to0^+}\frac{f(t)}{\varrho^p(t)}+\int_{(0,L)}f(t)\psi(t)\dnu(t)&\gtrsim \lim_{t\to0^+}\frac{f(t)h^p(t)}{\varrho^p(t)} + \frac{f(x_{K_-+1})h^p(x_{K_-+1})}{\varrho^p(x_{K_-+1})}\\
		&\quad+\sum_{k=K_-+2}^{K^+-2}\frac{f(x_k)h^p(x_k)}{\varrho^p(x_k)}\\
		&=\sum_{k\in\K^+_-}\frac{f(x_k)h^p(x_k)}{\varrho^p(x_k)}.
	\end{align*}
	The other two cases can be handled similarly.
\end{proof}

\begin{lemma}\label{lem:disantidissup} 
	Let $p\in(0,\infty)$ and $h\in Q_\varrho(0,L)$. Assume that there exist $C_1,\ C_2\in (0,\infty)$, $\alpha,\beta\in[0,\infty)$ and a nonnegative Borel measure $\nu$ on $(0,L)$ such that \eqref{thm:disantidis:repre} holds. Let $a>108\frac{C_2}{C_1}$ and $\{x_k\}_{k\in\K^+_-}\in CS(h,\varrho,a)$. Then for every $f\in\Mpl(0,L)$ we have
	\begin{equation}\label{lem:disantidissup:eq}
		\begin{aligned}
			\sum_{k\in\K^+}\left(\esssup_{t\in(x_{k-1},x_k]}\frac{h^\frac1{p}(t)}{\varrho^\frac1{p}(t)}f(t)\right)^p&\approx\sum_{k\in\K^+_-}h(x_k)\left(\esssup_{t\in(0,L)}\frac{f(t)}{\varrho^\frac1{p}(x_k)+\varrho^\frac1{p}(t)}\right)^p\\
			&\approx\alpha\left(\esssup_{t\in(0,L)}\frac{f(t)}{\varrho^\frac1{p}(t)}\right)^p+\beta\left(\esssup_{t\in(0,L)}f(t)\right)^p\\
			&\quad+\int_{(0,L)}\left(\esssup_{\tau\in(0,L)}\frac{\varrho^\frac1{p}(t)f(\tau)}{\varrho^\frac1{p}(t)+\varrho^\frac1{p}(\tau)}\right)^p\dnu(t).
		\end{aligned}
	\end{equation}
	Moreover, the equivalence constants in \eqref{lem:disantidissup:eq} depend only on $p, a, C_1, C_2$.
\end{lemma}
\begin{proof}
	We may clearly assume that $h\not\equiv0$ on $(0,L)$ (recall Convention~\ref{conv}(\ref{conv:indefinite})), and so $h\neq0$ on $(0,L)$. Note that 
$h^\frac1{p}\in Q_{\varrho^\frac1{p}}(0,L)$ and $\{x_k\}_{k\in\K^+_-}\in CS(h^\frac1{p},\varrho^\frac1{p},a^\frac1{p})$. The first equivalence follows from \citep[Theorem~4.2.7 and Remark 4.2.8]{EGO:18}.
	
	As for the second equivalence, observe that the function $t\mapsto\left(\esssup_{\tau\in(0,L)}\frac{\varrho^\frac1{p}(t)f(\tau)}{\varrho^\frac1{p}(t)+\varrho^\frac1{p}(\tau)}\right)^p$ is in $Q_{\varrho}(0,L)$ (we may assume that $f$ is finite a.e.~in $(0,L)$, otherwise \eqref{lem:disantidissup:eq} holds plainly). Consider Theorem~\ref{thm:disantidis} with the setting $\widetilde{p} = 1$, $\widetilde f(t)=\left(\esssup_{\tau\in(0,L)}\frac{\varrho^\frac1{p}(t)f(\tau)}{\varrho^\frac1{p}(t)+\varrho^\frac1{p}(\tau)}\right)^p$, and $\widetilde\varrho=\varrho$, where the symbols with tildes correspond to those from the statement of the theorem. It implies
	\begin{align*}
		&\sum_{k\in\K^+_-}h(x_k)\left(\esssup_{t\in(0,L)}\frac{f(t)}{\varrho^\frac1{p}(x_k)+\varrho^\frac1{p}(t)}\right)^p\\
		& \approx \alpha\lim_{t\to0^+}\left(\esssup_{\tau\in(0,L)}\frac{f(\tau)}{\varrho^\frac1{p}(t)+\varrho^\frac1{p}(\tau)}\right)^p+\beta\lim_{t\to L^-}\left(\esssup_{\tau\in(0,L)}\frac{\varrho^\frac1{p}(t)f(\tau)}{\varrho^\frac1{p}(t)+\varrho^\frac1{p}(\tau)}\right)^p\\
		&+\int_{(0,L)}\left(\esssup_{\tau\in(0,L)}\frac{\varrho^\frac1{p}(t)f(\tau)}{\varrho^\frac1{p}(t)+\varrho^\frac1{p}(\tau)}\right)^p\dnu(t),
	\end{align*}
	which is the second equivalence in \eqref{lem:disantidissup:eq} upon observing that
	\begin{align*}
		\lim_{t\to0^+}\left(\esssup_{\tau\in(0,L)}\frac{f(\tau)}{\varrho^\frac1{p}(t)+\varrho^\frac1{p}(\tau)}\right)^p&\approx\left(\esssup_{\tau\in(0,L)}\frac{f(\tau)}{\varrho^\frac1{p}(\tau)}\right)^p\\
		\intertext{and}
		\lim_{t\to L^-}\left(\esssup_{\tau\in(0,L)}\frac{\varrho^\frac1{p}(t)f(\tau)}{\varrho^\frac1{p}(t)+\varrho^\frac1{p}(\tau)}\right)^p&\approx\left(\esssup_{\tau\in(0,L)}f(\tau)\right)^p.
	\end{align*}
\end{proof}

\begin{lemma}\label{lem:disantidisint} 
	Let $p\in(0,\infty)$ and $h\in Q_\varrho(0,L)$. Assume that there exist $C_1,\ C_2\in (0,\infty)$, $\alpha,\beta\in[0,\infty)$ and a nonnegative Borel measure $\nu$ on $(0,L)$ such that \eqref{thm:disantidis:repre} holds. Let $a>108\frac{C_2}{C_1}$ and $\{x_k\}_{k\in\K^+_-}\in CS(h,\varrho,a)$. Then for every $f\in\Mpl(0,L)$ we have
	\begin{equation}\label{lem:disantidisint:eq}
		\begin{aligned}
			\sum_{k\in\K^+}\left(\int_{x_{k-1}}^{x_k}\frac{h^\frac1{p}(t)}{\varrho^\frac1{p}(t)}f(t) \dt\right)^p&\approx\sum_{k\in\K^+_-}h(x_k)\left(\int_0^L\frac{f(t)}{\varrho^\frac1{p}(x_k)+\varrho^\frac1{p}(t)} \dt\right)^p\\
			&\approx\alpha\left(\int_0^L\frac{f(t)}{\varrho^\frac1{p}(t)}\dt \right)^p+\beta\left(\int_0^Lf(t)\dt \right)^p\\
			&\quad+\int_{(0,L)}\left(\int_0^L\frac{\varrho^\frac1{p}(t)f(s)}{\varrho^\frac1{p}(t)+\varrho^\frac1{p}(s)}\ds \right)^p\dnu(t).
		\end{aligned}
	\end{equation}
	
	Moreover, the equivalence constants in \eqref{lem:disantidisint:eq} depend only on $p, a, C_1, C_2$.
\end{lemma}

\begin{proof}
	We omit the proof because it is similar to the proof of Lemma~\ref{lem:disantidissup}. We just note that the first equivalence in \eqref{lem:disantidisint:eq} follows from \citep[Theorem~4.2.5 and Remark 4.2.6]{EGO:18}.
\end{proof}

\begin{lemma}\label{lem:discantisup} 
	Let $p\in(0, \infty)$, $\varphi\in Q_\varrho(0,L)$ and $\{x_k\}_{k\in\K^+_-}\in CS(\varphi,\varrho,a)$ with $a>1$. For every $f\in\Mpl(0,L)$ we have
	\begin{equation}\label{lem:discantisup:eq}
		\begin{aligned}
			\sup_{t\in(0,L)}\varphi(t)\left(\int_0^L\frac{f(s)}{\varrho^\frac1{p}(t)+\varrho^\frac1{p}(s)}\ds \right)^p&\approx\sup_{k\in\K^+_-}\varphi(x_k)\left(\int_0^L\frac{f(s)}{\varrho^\frac1{p}(x_k)+\varrho^\frac1{p}(s)}\ds \right)^p\\
			&\approx\sup_{k\in\K^+}\left(\int_{x_{k-1}}^{x_k}f(s)\frac{\varphi^\frac1{p}(s)}{\varrho^\frac1{p}(s)}\ds \right)^p.
		\end{aligned}
	\end{equation}
	
	Moreover, the equivalence constants depend only on $p$ and $a$.
\end{lemma}
\begin{proof}
	We may clearly assume that $\varphi\not\equiv0$ on $(0,L)$, and so $\varphi\neq0$ on $(0,L)$. The second equivalence in \eqref{lem:discantisup:eq} follows from \citep[Theorem~4.2.5 and Remark~4.2.6]{EGO:18} (note that $\varphi^\frac1{p}\in Q_{\varrho^\frac1{p}}(0,L)$ and $\{x_k\}_{k\in\K^+_-}\in CS(\varphi^\frac1{p},\varrho^\frac1{p},a^\frac1{p}$)) upon observing that
	\begin{align*}
		\sup_{k\in\K^+_-}\varphi(x_k)\left(\int_0^L\frac{f(s)}{\varrho^\frac1{p}(x_k)+\varrho^\frac1{p}(s)}\ds \right)^p &= \left(\sup_{k\in\K^+_-}\varphi^\frac1{p}(x_k)\int_0^L\frac{f(s)}{\varrho^\frac1{p}(x_k)+\varrho^\frac1{p}(s)}\ds \right)^p\\
		&\approx\left(\sup_{k\in\K^+_-}\frac{\varphi^\frac1{p}(x_k)}{\varrho^\frac1{p}(x_k)}\int_0^L\min\{\varrho^\frac1{p}(x_k),\varrho^\frac1{p}(s)\}\dnu(s)\right)^p,
	\end{align*}
	where $\d\nu(s)=f(s)\varrho^{-\frac1{p}}(s)\ds $.
	
	We shall prove the first equivalence in \eqref{lem:discantisup:eq}. Using \eqref{prel:CSZ1} and \eqref{prel:CSZ2}, we have
	\begin{align}
		&\hspace{-20pt}\sup_{t\in(0,L)}\varphi(t)\left(\int_0^L\frac{f(s)}{\varrho^\frac1{p}(t)+\varrho^\frac1{p}(s)}\ds \right)^p \notag\\
		&=\sup_{k\in\K^+}\sup_{t\in{(x_{k-1},x_k]}}\varphi(t)\left(\int_0^L\frac{f(s)}{\varrho^\frac1{p}(t) + \varrho^\frac1{p}(s)}\ds \right)^p \notag\\
		&\approx\sup_{k\in\mathcal{Z}_1}\sup_{t\in{(x_{k-1},x_k]}}\varphi(t)\left(\int_0^L\frac{f(s)}{\varrho^\frac1{p}(t) + \varrho^\frac1{p}(s)}\ds \right)^p
		+\sup_{k\in\mathcal{Z}_2}\sup_{t\in{(x_{k-1},x_k]}}\varphi(t)\left(\int_0^L\frac{f(s)}{\varrho^\frac1{p}(t)+\varrho^\frac1{p}(s)}\ds \right)^p \notag\\
		&\approx\sup_{k\in\mathcal{Z}_1}\varphi(x_{k-1})\sup_{t\in{(x_{k-1},x_k]}}\left(\int_0^L\frac{f(s)}{\varrho^\frac1{p}(t)+\varrho^\frac1{p}(s)}\ds \right)^p
		+\sup_{k\in\mathcal{Z}_2}\frac{\varphi(x_k)}{\varrho(x_k)}\sup_{t\in{(x_{k-1},x_k]}}\left(\int_0^L\frac{\varrho^\frac1{p}(t)f(s)}{\varrho^\frac1{p}(t)+\varrho^\frac1{p}(s)}\ds \right)^p \notag\\
		&=\sup_{k\in\mathcal{Z}_1}\varphi(x_{k-1})\left(\int_0^L\frac{f(s)}{\varrho^\frac1{p}(x_{k-1}) + \varrho^\frac1{p}(s)}\ds \right)^p
		+\sup_{k\in\mathcal{Z}_2}\varphi(x_k)\left(\int_0^L\frac{f(s)}{\varrho^\frac1{p}(x_k)+\varrho^\frac1{p}(s)}\ds \right)^p,\label{lem:discantisup:eq1}
	\end{align}
	where $\K^+=\mathcal{Z}_1\cup\mathcal{Z}_2$ is a decomposition of $\K^+$ from \eqref{prel:CSdecomp}. Note that the second equivalence in \eqref{lem:discantisup:eq1} is valid even when $K_-+1\in\mathcal{Z}_1$ or $K_+\in\mathcal{Z}_2$. Indeed, if $K_-+1\in\mathcal{Z}_1$ (and so $K_->-\infty$), we have $\varphi(x_{K_-})=\lim_{t\to0^+}\varphi(t)\approx\varphi(x_{K_-+1})>0$ thanks to \eqref{prel:CSZ1}. Hence,
	\begin{align*}
		\sup_{t\in{(0,x_{K_-+1}]}}\varphi(t)\left(\int_0^L\frac{f(s)}{\varrho^\frac1{p}(t)+\varrho^\frac1{p}(s)}\ds \right)^p&\approx\left(\lim_{t\to0^+}\varphi(t)\right)\left(\sup_{t\in{(0,x_{K_-+1}]}}\left(\int_0^L\frac{f(s)}{\varrho^\frac1{p}(t) + \varrho^\frac1{p}(s)}\ds \right)^p\right)\\
		&=\varphi(x_{K_-})\left(\sup_{t\in{(0,x_{K_-+1}]}}\left(\int_0^L\frac{f(s)}{\varrho^\frac1{p}(t) + \varrho^\frac1{p}(s)}\ds \right)^p\right).
	\end{align*}
	Analogously, one may show that, if $K_+\in\mathcal{Z}_2$, then
	\begin{equation*}
		\sup_{t\in{(x_{K_+-1},L)}}\varphi(t)\left(\int_0^L\frac{f(s)}{\varrho^\frac1{p}(t)+\varrho^\frac1{p}(s)}\ds \right)^p\approx\varphi(x_{K_+})\sup_{t\in{(x_{K_+-1},L)}}\left(\int_0^L\frac{f(s)}{\varrho^\frac1{p}(t)+\varrho^\frac1{p}(s)}\ds \right)^p.
	\end{equation*}
	Next, one clearly has
	\begin{align}
		&\sup_{k\in\mathcal{Z}_1}\varphi(x_{k-1})\left(\int_0^L\frac{f(s)}{\varrho^\frac1{p}(x_{k-1})+\varrho^\frac1{p}(s)}\ds\right)^p
		+\sup_{k\in\mathcal{Z}_2}\varphi(x_k)\left(\int_0^L\frac{f(s)}{\varrho^\frac1{p}(x_k) + \varrho^\frac1{p}(s)}\ds \right)^p \notag\\
		&\qquad\lesssim\sup_{k\in\K^+_-}\varphi(x_k)\left(\int_0^L\frac{f(s)}{\varrho^\frac1{p}(x_k) + \varrho^\frac1{p}(s)}\ds \right)^p. \label{lem:discantisup:eq2}
	\end{align}
	For any $k\in\mathcal{Z}_1$ we have
	\begin{align*}
		\varphi(x_k)\left(\int_0^L\frac{f(s)}{\varrho^\frac1{p}(x_k)+\varrho^\frac1{p}(s)}\ds \right)^p & \approx\varphi(x_{k-1})\left(\int_0^L\frac{f(s)}{\varrho^\frac1{p}(x_k) + \varrho^\frac1{p}(s)}\ds \right)^p \\
		& \le\varphi(x_{k-1})\left(\int_0^L\frac{f(s)}{\varrho^\frac1{p}(x_{k-1}) + \varrho^\frac1{p}(s)}\ds \right)^p.
	\end{align*}
	Hence,
	\begin{align}
		&\sup_{k\in\K^+}\varphi(x_k)\left(\int_0^L\frac{f(s)}{\varrho^\frac1{p}(x_k)+\varrho^\frac1{p}(s)}\ds \right)^p \notag\\
		&\quad \lesssim\sup_{k\in\mathcal{Z}_1}\varphi(x_{k-1})\left(\int_0^L\frac{f(s)}{\varrho^\frac1{p}(x_{k-1})+\varrho^\frac1{p}(s)}\ds \right)^p+\sup_{k\in\mathcal{Z}_2}\varphi(x_k)\left(\int_0^L\frac{f(s)}{\varrho^\frac1{p}(x_k)+\varrho^\frac1{p}(s)}\ds \right)^p. \label{lem:discantisup:eq3}
	\end{align}
	If $K_-=-\infty$ (and so $\K_-^+=\K^+$), we obtain
	\begin{align}
		&\sup_{k\in\mathcal{Z}_1}\varphi(x_{k-1})\left(\int_0^L\frac{f(s)}{\varrho^\frac1{p}(x_{k-1})+\varrho^\frac1{p}(s)}\ds \right)^p +\sup_{k\in\mathcal{Z}_2}\varphi(x_k)\left(\int_0^L\frac{f(s)}{\varrho^\frac1{p}(x_k)+\varrho^\frac1{p}(s)}\ds \right)^p \notag\\
		&\quad \approx\sup_{k\in\K^+_-}\varphi(x_k)\left(\int_0^L\frac{f(s)}{\varrho^\frac1{p}(x_k)+\varrho^\frac1{p}(s)}\ds \right)^p \label{lem:discantisup:eq4}
	\end{align}
	by combining \eqref{lem:discantisup:eq2} with \eqref{lem:discantisup:eq3}.
	
	Now suppose that $K_->-\infty$. If $K_-+1\in\mathcal{Z}_2$, then $\lim_{t\to0^+}\frac{\varphi(t)}{\varrho(t)}\approx\frac{\varphi(x_{K_-+1})}{\varrho(x_{K_-+1})}\in(0,\infty)$ thanks to \eqref{prel:CSZ2}, and so
	\begin{align*}
		\lim_{t\to0^+}\varphi(t)\left(\int_0^L\frac{f(s)}{\varrho^\frac1{p}(t)+\varrho^\frac1{p}(s)}\ds \right)^p&\le\left(\lim_{t\to0^+}\frac{\varphi(t)}{\varrho(t)}\right)\left(\sup_{t\in(0,x_{K_-+1}]}\left(\int_0^L\frac{\varrho^\frac1{p}(t)f(s)}{\varrho^\frac1{p}(t)+\varrho^\frac1{p}(s)}\ds \right)^p\right)\\
		&\approx\frac{\varphi(x_{K_-+1})}{\varrho(x_{K_-+1)}}\left(\int_0^L\frac{\varrho^\frac1{p}(x_{K-+1})f(s)}{\varrho^\frac1{p}(x_{K-+1})+\varrho^\frac1{p}(s)}\ds \right)^p\\
		&\leq \sup_{k\in\mathcal{Z}_2}\varphi(x_k)\left(\int_0^L\frac{f(s)}{\varrho^\frac1{p}(x_k)+\varrho^\frac1{p}(s)}\ds \right)^p.
	\end{align*}
	If $K_-+1\in\mathcal{Z}_1$, we plainly have
	\begin{equation*}
		\lim_{t\to0^+}\varphi(t)\left(\int_0^L\frac{f(s)}{\varrho^\frac1{p}(t)+\varrho^\frac1{p}(s)}\ds \right)^p\leq\sup_{k\in\mathcal{Z}_1}\varphi(x_{k-1})\left(\int_0^L\frac{f(s)}{\varrho^\frac1{p}(x_{k-1})+\varrho^\frac1{p}(s)}\ds \right)^p.
	\end{equation*}
	Hence, whether $K_-+1\in\mathcal{Z}_1$ or $K_-+1\in\mathcal{Z}_2$, we obtain
	\begin{align*}
		&\lim_{t\to 0^+}\varphi(t)\left(\int_0^L\frac{f(s)}{\varrho^\frac1{p}(t)+\varrho^\frac1{p}(s)}\ds \right)^p\\
		& \quad \lesssim \sup_{k\in\mathcal{Z}_1}\varphi(x_{k-1})\left(\int_0^L\frac{f(s)}{\varrho^\frac1{p}(x_{k-1})+\varrho^\frac1{p}(s)}\ds \right)^p
		+\sup_{k\in\mathcal{Z}_2}\varphi(x_k)\left(\int_0^L\frac{f(s)}{\varrho^\frac1{p}(x_k)+\varrho^\frac1{p}(s)}\ds \right)^p.
	\end{align*}
	Therefore, combining the last inequality with \eqref{lem:discantisup:eq2} and \eqref{lem:discantisup:eq3}, we obtain equivalence \eqref{lem:discantisup:eq4} even when $K_->-\infty$.
	
	Finally, the first equivalence in \eqref{lem:discantisup:eq} follows by combining \eqref{lem:discantisup:eq1} with \eqref{lem:discantisup:eq4}.
\end{proof}

\section{Main results}

We are finally ready to present our main results. The first one is the desired characterization of \eqref{I:main} when all the involved exponents are finite.

\begin{theorem}\label{thm:HanickaToLambda}
        Let $p,q\in(0,\infty)$. Let $v,w$ be weights on $(0,L)$ and $u$ an a.e.~positive weight on $(0,L)$.
        Set
        \begin{equation*}
                C = \sup_{\|f\|_{\Gpuv}\le1}\|f\|_{\Lambda^{q}(w)}.
        \end{equation*}
        \begin{enumerate}[\rm (i)]
                \item
                If $1\le q$ and $p\le q<\infty$, then $C\approx A_1$, where
                \begin{equation*}
                        A_1 = \sup_{0<t<L}\frac{W^\frac1{q}(t)}{\left(V(t)+U^p(t)\int_t^Lv(s)U^{-p}(s)\ds \right)^\frac1{p}}.
                \end{equation*}
                \item
                If $1\le q < p<\infty$, then $C\approx A_2$, where
                \begin{equation*}
                        \begin{aligned}
                                A_2 & = \left(\int_0^L\frac{V(t)\int_t^Lv(s)U^{-p}(s)\ds \,U^{\frac{pq}{p-q}+p-1}(t)u(t)\sup_{\tau\in[t,L)}U^{-\frac{pq}{p-q}}(\tau)W^{\frac{p}{p-q}}(\tau)}{\left(V(t)+U^p(t)\int_t^Lv(s)U^{-p}(s)\ds \right)^{\frac{q}{p-q}+2}} \dt\right)^{\frac{p-q}{pq}}\\
                                &\quad+\left(\lim_{t\to0^+}\frac{U^{p}(t)}{V(t)+U^p(t)\int_{t}^Lv(s)U^{-p}(s)\ds }\right)^\frac1{p}\left(\sup_{t\in(0,L)}\frac{W(t)}{U^q(t)}\right)^\frac1{q}\\
                                &\quad+\left(\lim_{t\to L^-}\frac1{V(t)+U^p(t)\int_{t}^Lv(s)U^{-p}(s)\ds }\right)^\frac1{p}W^\frac1{q}(L).
                        \end{aligned}
                \end{equation*}
                \item
                If $p\le q<1$, then $C\approx A_3$, where
                \begin{equation*}
                        A_3 = \sup_{0<t<L}\frac{W^\frac1{q}(t)+U(t)\left(\int_t^LW^\frac{q}{1-q}(s)w(s)U^{-\frac{q}{1-q}}(s)\ds \right)^\frac{1-q}{q}}{\left(V(t)+U^p(t)\int_t^Lv(s)U^{-p}(s)\ds \right)^\frac1{p}}.
                \end{equation*}
                \item
                If $q<1$ and $q<p<\infty$, then $C\approx A_4$, where
                \begin{align*}
                        A_4 &=\left(\lim_{t\to0^+}\frac{U^p(t)}{V(t)+U^p(t)\int_{t}^Lv(s)U^{-p}(s)\ds }\right)^\frac1{p}\left(\int_0^LW^\frac{q}{1-q}(t)w(t)U^{-\frac{q}{1-q}}(t) \dt\right)^\frac{1-q}{q}\\
                        &\quad+\left(\lim_{t\to L^-}\frac1{V(t)+U^p(t)\int_{t}^Lv(s)U^{-p}(s)\ds }\right)^\frac1{p}\left(\int_0^LW^\frac{q}{1-q}(t)w(t) \dt\right)^\frac{1-q}{q}\\
                        &\quad+\left( \int_0^L\frac{\left(W^\frac{1}{1-q}(t)+U^\frac{q}{1-q}(t)\int_t^LW^\frac{q}{1-q}(s)w(s)U^{-\frac{q}{1-q}}(s)\ds \right)^\frac{p(1-q)}{p-q} } {\left(V(t)+U^p(t)\int_{t}^Lv(s)U^{-p}(s)\ds \right)^{\frac{q}{p-q}+2}} \right.\\
                        &\hspace{200pt} \times \left. \vphantom{\left( \int_0^L\frac{\left(W^\frac{1}{1-q}(t)+U^\frac{q}{1-q}(t)\int_t^LW^\frac{q}{1-q}(s)w(s)U^{-\frac{q}{1-q}}(s)\ds \right)^\frac{p(1-q)}{p-q} } {\left(V(t)+U^p(t)\int_{t}^Lv(s)U^{-p}(s)\ds \right)^{\frac{q}{p-q}+2}} \right.}
                                V(t)U^{p-1}(t)u(t)\int_t^Lv(s)U^{-p}(s)\ds \dt \right)^\frac{p-q}{pq}.
                \end{align*}
        \end{enumerate}
        The equivalence constants depend only on the parameters $p$ and $q$. In particular, they are independent of the weights $u$, $v$ and $w$.
\end{theorem}

\begin{proof}
        First of all, note that $U$ is admissible. Furthermore, as a~prelude to the proof, let us make the following observation. Suppose that there exists a~$t_0\in(0,L)$ such that $\int_{t_0}^L v(s) U^{-p}(s)\ds = \infty$. Then the same holds, in fact, for all $t\in (0,L)$ (if $t>t_0$, consider that $v$ is locally integrable and $U$ is admissible; thus $\int_{t}^L v(s) U^{-p}(s)\ds = \infty$ must hold as well). It follows that $\Gpuv=\{0\}$, where ``$0$'' is the zero-constant function.
        Therefore, $C=0$ and, by Convention \ref{conv}(\ref{conv:indefinite}), the quantities $A_1$--$A_4$ are also equal to zero; hence the theorem holds trivially. Thanks to this observation, we may and will assume in the proof that
        \begin{equation}\label{thm:HanickaToLambda:assumV}
                \int_{t}^L\frac{v(s)}{U^p(s)}\ds <\infty \quad \text{ for 
every } t\in(0,L).
        \end{equation}                     
        If $p>q$, set $r=\frac{pq}{p-q}$. For each $f\in\M_\mu (X)$ there exists a sequence $\{h_n\}_{n\in\N}$ of functions from $\Mpl(0,L)$ such that $\int_t^L h_n(s)\ds \nearrow f^*(t)$ for a.e.~$t\in(0,L)$ as $n\to\infty$. The proof of this statement 
is analogous to that of \cite[Lemma 1.2]{Sinnamon2003}.
        Furthermore, for any $t\in(0,L)$ and every $h\in\Mpl(0,L)$, we have
        \begin{equation*}
                \frac1{U(t)} \int_0^t u(y)\int_y^L h(s) \ds \dy \le 2 \int_0^L \frac{U(s) h(s)}{U(s)+U(t)} \ds \le \frac2{U(t)} \int_0^t u(y)\int_y^L h(s) \ds \dy.
        \end{equation*}
        Hence, by the monotone convergence theorem, we get (Convention \ref{conv}(\ref{conv:indefinite}) is in use here)
        \begin{equation}\label{thm:HanickaToLambda:equivalentexpression}
                C \approx \sup_{h\in\Mpl(0,L)} \frac{ \left(\int_0^L\left(\int_t^L h(s)\ds \right)^q w(t) \dt\right)^\frac1{q} } { \left(\int_0^L\left(\int_0^L\frac{U(s)h(s)}{U(s)+U(t)}\ds \right)^pv(t) \dt\right)^\frac1{p} }.
        \end{equation}
        Define
        \begin{equation*}
                \varphi(t)=\int_0^L\min\{U(t)^p, U(s)^p\}\frac{v(s)}{U(s)^p}\ds,\ t\in(0,L).
        \end{equation*}
        Note that $\varphi\in Q_{U^p}(0,L)$ (in particular, $\varphi$ is finite on $(0,L)$ by assumption \eqref{thm:HanickaToLambda:assumV}). Therefore, for every $a>1$ there exists a~covering sequence $\{x_k\}_{k\in\K^+_-}\in CS(\varphi, U^p, a)$. We fix $a>1$ sufficiently large so that the 
lemmas and theorems that we are to use below may be applied. An~appropriate value of $a$ may be determined by inspecting the further course of the 
proof in each of the cases (i)--(iv). In any of them, however, the sufficient size of the parameter $a$ depends only on $p$ and $q$.
       
        By Lemma~\ref{lem:disantidisint} with $\widetilde p=p$, $\widetilde h=\varphi$, $\widetilde\varrho=U^p$, $\widetilde f(t)=Uh$, $\widetilde\alpha=\widetilde\beta=0$, and $\d\widetilde \nu(t)=\frac{v(t)}{U^p(t)} \dt$ (the parameters with tildes are those from the lemma) we have
        \begin{align}
                \int_0^L\left(\int_0^L\frac{U(s)h(s)}{U(s)+U(t)}\ds \right)^p v(t) \dt
                &\approx\sum_{k\in\K_-^+}\varphi(x_k)\left(\int_0^L \frac{U(t)h(t)}{U(x_k)+U(t)} \dt\right)^p \notag\\
                &\approx\sum_{k\in\K^+}\left(\int_{x_{k-1}}^{x_k}\varphi^\frac1{p}(t)h(t) \dt\right)^p \label{thm:HanickaToLambda:disantidisRHS}
        \end{align}
        for each $h\in\Mpl(0, L)$. The equivalence constants depend only on $p$. Clearly,
        \begin{align}
                &\int_0^L\left(\int_t^L h(s)\ds \right)^qw(t) \dt=\sum_{k\in\K^+}\int_{x_{k-1}}^{x_k}\left(\int_t^L h(s)\ds \right)^q w(t) \dt \notag\\
                &\quad\approx \sum_{k\in\K^+}\int_{x_{k-1}}^{x_k}\left(\int_t^{x_k} h(s)\ds \right)^q w(t) \dt + \sum_{k\in\K}\left(\int_{x_k}^L h(s)\ds \right)^q\int_{x_{k-1}}^{x_k} w(t) \dt \label{thm:HanickaToLambda:disantidisLHS}
        \end{align}
        for each $h\in\Mpl(0, L)$, and the equivalence constants depend only on $q$.\\
       
        \emph{Upper bounds}.
        In this part, we shall prove the upper bounds on $C$. This is equivalent to proving that the upper bounds are upper bounds on the supremum on the right-hand side of \eqref{thm:HanickaToLambda:equivalentexpression}. As for cases (i) and (ii), assume that $1\le q<\infty$. For every $k\in\K^+$, the weighted Hardy inequality (e.g., \citep{OK:90} and references therein) yields
        \begin{equation}\label{thm:HanickaToLambda:hardygeq1}
                \int_{x_{k-1}}^{x_k}\left(\int_t^{x_k} h(s)\ds \right)^q w(t) \dt\lesssim\left(\int_{x_{k-1}}^{x_k}\varphi^\frac1{p}(t)h(t) \dt\right)^q\sup_{t\in(x_{k-1},x_k]}\varphi^{-\frac{q}{p}}(t)\int_{x_{k-1}}^t w(s)\ds .
        \end{equation}
       
        \emph{Case} (i).
        Assume that $1\leq q$, $p\le q$ and $A_1<\infty$. Then
        \begin{align}
                & \sum_{k\in\K^+}\int_{x_{k-1}}^{x_k}\left(\int_t^{x_k} h(s)\ds \right)^q w(t) \dt
                \lesssim\left(\sup_{t\in(0,L)}\varphi^{-\frac{q}{p}}(t)W(t)\right)\sum_{k\in\K^+}\left(\int_{x_{k-1}}^{x_k}\varphi^\frac1{p}(t)h(t) \dt\right)^q \notag\\
                &\qquad\le A_1^{q}\left(\sum_{k\in\K^+}\left(\int_{x_{k-1}}^{x_k}\varphi^\frac1{p}(t)h(t) \dt\right)^p\right)^\frac{q}{p}
                \approx A_1^{q}\left(\int_0^L\left(\int_0^L\frac{U(s)h(s)}{U(s)+U(t)}\ds \right)^p v(t) \dt\right)^\frac{q}{p}. \label{thm:HanickaToLambda:suffAI}
        \end{align}
        The first inequality in \eqref{thm:HanickaToLambda:suffAI} follows from \eqref{thm:HanickaToLambda:hardygeq1}, the second inequality is valid since $p\le q$, and the equivalence is valid thanks to \eqref{thm:HanickaToLambda:disantidisRHS}. Furthermore, using $p\le q$, we get
        \begin{align}
                \sum_{k\in\K}\left(\int_{x_k}^L h(s)\ds \right)^q\int_{x_{k-1}}^{x_k} w(t) \dt
                &\le\sum_{k\in\K}\left(\int_{x_k}^L h(s)\ds \right)^q\varphi^\frac{q}{p}(x_k)\varphi^{-\frac{q}{p}}(x_k)W(x_k) \notag\\
                &\le\sup_{k\in\K}\left(\varphi^{-\frac{q}{p}}(x_k)W(x_k)\right)\left(\sum_{k\in\K}\left(\int_{x_k}^L h(s)\ds \right)^p\varphi(x_k)\right)^\frac{q}{p} \notag\\
                &\lesssim A_1^{q}\left(\sum_{k\in\K}\left(\int_{x_k}^L \frac{U(s)h(s)}{U(x_k)+U(s)}\ds \right)^p\varphi(x_k)\right)^\frac{q}{p} \notag\\
                &\lesssim A_1^{q}\left(\int_0^L\left(\int_0^L\frac{U(s)h(s)}{U(s)+U(t)}\ds \right)^p v(t) \dt\right)^\frac{q}{p}, \label{thm:HanickaToLambda:suffAII}
        \end{align}
        where the last inequality follows from \eqref{thm:HanickaToLambda:disantidisRHS}. Note that \eqref{thm:HanickaToLambda:suffAII} is actually valid for any $q\in(0,\infty)$ such that $p\leq q$. By combining \eqref{thm:HanickaToLambda:disantidisLHS}, \eqref{thm:HanickaToLambda:suffAI}, \eqref{thm:HanickaToLambda:suffAII}, and considering \eqref{thm:HanickaToLambda:equivalentexpression}, we obtain the estimate $C\lesssim A_1$ in case (i).
       
        \emph{Case} (ii). Assume that $1\le q < p< \infty$ and $A_2<\infty$. Owing to the H\"older  inequality with exponents $\frac{p}{q}$ and $\frac{p}{p-q}$, we obtain
        \begin{align}
                &\sum_{k\in\K^+}\left(\int_{x_{k-1}}^{x_k}\varphi^\frac1{p}(t)h(t) \dt\right)^q\sup_{t\in(x_{k-1},x_k]}\varphi^{-\frac{q}{p}}(t) \int_{x_{k-1}}^t w(s)\ds \notag\\
                &\quad\le\left(\sum_{k\in\K^+}\left(\int_{x_{k-1}}^{x_k}\varphi^\frac1{p}(t)h(t) \dt\right)^p\right)^\frac{q}{p} \left(\sum_{k\in\K^+}\sup_{t\in(x_{k-1},x_k]}\varphi^{-\frac{r}{p}}(t)\left(\int_{x_{k-1}}^t w(s)\ds \right)^\frac{r}{q} \right)^\frac{p-q}{p} \notag\\
                &\quad\le\left(\sum_{k\in\K^+}\left(\int_{x_{k-1}}^{x_k}\varphi^\frac1{p}(t)h(t) \dt\right)^p\right)^\frac{q}{p} \left(\sum_{k\in\K^+}\sup_{t\in(x_{k-1},x_k]}\left(\frac{U^p(t)}{\varphi(t)}\right)^{\frac{r}{p}}U^{-r}(t)W(t)^\frac{r}{q} \right)^\frac{p-q}{p}. \label{thm:HanickaToLambda:suffBIeq1}
        \end{align}
        By \citep[Theorem~2.4.4]{EGO:18}, the equivalence
        \begin{equation}\label{thm:HanickaToLambda:suffBIeq3}
                \begin{aligned}
                        \left(\frac{U^p(t)}{\varphi(t)}\right)^{\frac{r}{p}}
                        &\approx\left(\lim_{s\to0^+}\frac{U^p(s)}{\varphi(s)}\right)^\frac{r}{p}+\left(\lim_{s\to L^-}\frac1{\varphi(s)}\right)^\frac{r}{p}U^r(t)\\
                        &\quad+\int_0^L\min\{U^r(t),U^r(s)\} \frac{U^{p-1}(s)u(s)V(s)} {\varphi^{\frac{r}{p}+2}(s)} \int_s^L\frac{v(\tau)}{U^p(\tau)}\,\d\tau\ds
                \end{aligned}
        \end{equation}
        is valid for every $t\in(0,L)$, and the equivalence constants depend only on $p$ and $q$. Note that \eqref{thm:HanickaToLambda:suffBIeq3} is actually valid for any $q\in(0,\infty)$ such that $q<p$. Lemma~\ref{lem:disantidissup} with the setting $\widetilde h=\left(\frac{U^p}{\varphi}\right)^{\frac{r}{p}}$, $\widetilde\varrho=U^r$, $\widetilde p=1$, $\widetilde f =W^\frac{r}{q}$, together with \eqref{thm:HanickaToLambda:suffBIeq3} gives
        \begin{equation}\label{thm:HanickaToLambda:suffBIeq4}
                \sum_{k\in\K^+}\sup_{t\in(x_{k-1},x_k]}\left(\frac{U^p(t)}{\varphi(t)}\right)^{\frac{r}{p}}U^{-r}(t)W(t)^\frac{r}{q}\approx A_2^{r}.
        \end{equation}
        By using \eqref{thm:HanickaToLambda:suffBIeq1}, \eqref{thm:HanickaToLambda:suffBIeq4} and \eqref{thm:HanickaToLambda:disantidisRHS} we obtain
        \begin{align}
                &\sum_{k\in\K^+}\left(\int_{x_{k-1}}^{x_k}\varphi^\frac1{p}(t)h(t) \dt\right)^q\sup_{t\in(x_{k-1},x_k]}\varphi^{-\frac{q}{p}}(t)\int_{x_{k-1}}^t w(s)\ds \notag \\
                &\quad\lesssim A_2^{q}\left(\int_0^L\left(\int_0^L\frac{U(s)h(s)}{U(s)+U(t)}\ds \right)^p v(t) \dt\right)^\frac{q}{p}. \label{thm:HanickaToLambda:suffBIeq5}
        \end{align}
        Next, one has
        \begin{align}
                &\sum_{k\in\K}\left(\int_{x_k}^L h(s)\ds \right)^q\int_{x_{k-1}}^{x_k} w(t) \dt \notag\\
                &\quad \le \sum_{k\in\K}\left(\int_{x_k}^L h(s)\ds \right)^q\varphi^\frac{q}{p}(x_k)\varphi^{-\frac{q}{p}}(x_k)W(x_k) \notag\\
                &\quad \le \left(\sum_{k\in\K}\left(\int_{x_k}^L h(s)\ds \right)^p\varphi(x_k)\right)^\frac{q}{p} \left(\sum_{k\in\K}\varphi^{-\frac{r}{p}}(x_k)W^\frac{r}{q}(x_k)\right)^\frac{q}{r} \notag\\
                &\quad =\left(\sum_{k\in\K}\left(\sum_{\substack{l\in\Z\\ k\le l\le K^+-1}}\int_{x_l}^{x_{l+1}} h(s)\ds \right)^p\varphi(x_k)\right)^\frac{q}{p} \left(\sum_{k\in\K}\varphi^{-\frac{r}{p}}(x_k)W^\frac{r}{q}(x_k)\right)^\frac{q}{r} \notag\\
                &\quad \approx\left(\sum_{k\in\K}\left(\int_{x_k}^{x_{k+1}} h(s)\ds \right)^p\varphi(x_k)\right)^\frac{q}{p} \left(\sum_{k\in\K}\varphi^{-\frac{r}{p}}(x_k)W^\frac{r}{q}(x_k)\right)^\frac{q}{r} \notag\\
                &\quad \le\left(\sum_{k\in\K}\left(\int_{x_k}^{x_{k+1}} \varphi^\frac1{p}(s)h(s)\ds \right)^p\right)^\frac{q}{p} \left(\sum_{k\in\K}\varphi^{-\frac{r}{p}}(x_k)W^\frac{r}{q}(x_k)\right)^\frac{q}{r} \notag\\
                &\quad \lesssim \left(\int_0^L\left(\int_0^L\frac{U(s)h(s)}{U(s)+U(t)}\ds \right)^p v(t) \dt\right)^\frac{q}{p} \left(\sum_{k\in\K}\varphi^{-\frac{r}{p}}(x_k)W^\frac{r}{q}(x_k)\right)^\frac{q}{r}. \label{thm:HanickaToLambda:suffBIeq6}
        \end{align}
        Here, the H\"older inequality was applied in the second step, the 
fourth step relies on \eqref{prel:CSgeomincr} and \citep[Lemma~1.3.5]{EGO:18}, and the last step follows from \eqref{thm:HanickaToLambda:disantidisRHS}.
        Note that \eqref{thm:HanickaToLambda:suffBIeq6} is actually valid 
for all $0<q<p<\infty$, although we are currently assuming $1\le q < p <\infty$.
        Since
        \begin{equation*}
                W(x_k)\le U^q(x_k)\sup_{t\in(x_k,L)}U^{-q}(t)W(t)
        \end{equation*}
        holds for each $k\in\K$, by \eqref{thm:HanickaToLambda:suffBIeq5} 
and \eqref{thm:HanickaToLambda:suffBIeq6} we get
        \begin{equation}\label{thm:HanickaToLambda:suffBIeq7}
                \sum_{k\in\K}\left(\int_{x_k}^L h(s)\ds \right)^q\int_{x_{k-1}}^{x_k} w(t) \dt\lesssim A_2^{q}\left(\int_0^L\left(\int_0^L\frac{U(s)h(s)}{U(s)+U(t)}\ds \right)^p v(t) \dt\right)^\frac{q}{p}.
        \end{equation}
        Using \eqref{thm:HanickaToLambda:disantidisLHS}, \eqref{thm:HanickaToLambda:hardygeq1}, \eqref{thm:HanickaToLambda:suffBIeq5}, \eqref{thm:HanickaToLambda:suffBIeq7}  and considering \eqref{thm:HanickaToLambda:equivalentexpression}, we obtain the estimate $C\lesssim A_2$ in case (ii).
       
        As for cases (iii) and (iv), assume that  $0 < q < 1$. One can easily modify \citep[Theorem~3.3]{SS:96} to obtain
        \begin{align}
                &\int_{x_{k-1}}^{x_k}\left(\int_t^{x_k}h(s)\ds \right)^q w(t) \dt \notag\\
                &\quad\lesssim\left(\int_{x_{k-1}}^{x_k}h(t)\varphi^\frac1{p}(t) \dt\right)^q\left(\int_{x_{k-1}}^{x_k}\left(\int_{x_k}^tw(s)\ds \right)^\frac{q}{1-q}w(t)\varphi^{-\frac{q}{p(1-q)}}(t) \dt\right)^{1-q} \notag\\
                &\quad\le\left(\int_{x_{k-1}}^{x_k}h(t)\varphi^\frac1{p}(t) \dt\right)^q\left(\int_{x_{k-1}}^{x_k}W^\frac{q}{1-q}(t)w(t)\varphi^{-\frac{q}{p(1-q)}}(t) \dt\right)^{1-q} \label{thm:HanickaToLambda:suffCIeq8}
        \end{align}
        for every $k\in\K^+$, where the constant in ``$\lesssim$'' depends only on $q$.
        By Lemma~\ref{lem:discantisup} with the setting $\widetilde\varphi(t)=U^q\varphi^{-\frac{q}{p}}$, $\widetilde f= W^\frac{q}{1-q}w$, $\widetilde p=1-q$, $\widetilde\varrho=U^q$, and by Lemma~\ref{lemma:changeofvariable} one has
        \begin{align}
                &\sup_{k\in\K^+}\left(\int_{x_{k-1}}^{x_k}W^\frac{q}{1-q}(t)w(t)\varphi^{-\frac{q}{p(1-q)}}(t) \dt\right)^{1-q} \notag\\
                &\quad\approx\sup_{t\in(0,L)} \frac{\left(\int_0^L W^\frac{q}{1-q}(s)w(s) \min\left\{1, \left( \frac{U(t)}{U(s)} \right)^\frac{q}{1-q} \right\}\ds \right)^{1-q}}{\varphi^\frac{q}{p}(t)}
                \approx A_3^{q}. \label{thm:HanickaToLambda:suffCIeq9}
        \end{align}
       
        \emph{Case} (iii). Assume that $0<p\le q<1$ and $A_3<\infty$. Thanks to \eqref{thm:HanickaToLambda:disantidisRHS}, \eqref{thm:HanickaToLambda:suffCIeq8} and \eqref{thm:HanickaToLambda:suffCIeq9}, we 
have
        \begin{align}
                &\sum_{k\in\K^+}\int_{x_{k-1}}^{x_k}\left(\int_t^{x_k} h(s)\ds \right)^q w(t) \dt \notag\\
                &\quad\lesssim\left(\sum_{k\in\K^+}\left(\int_{x_{k-1}}^{x_k}h(t)\varphi^\frac1{p}(t) \dt\right)^q\right)\left(\sup_{k\in\K^+} \left(\int_{x_{k-1}}^{x_k}W^\frac{q}{1-q}(t)w(t)\varphi^{-\frac{q}{p(1-q)}}(t) \dt\right)^{1-q}\right) \notag\\
                &\quad\lesssim A_3^{q}\left(\int_0^L\left(\int_0^L\frac{U(s)h(s)}{U(s)+U(t)}\ds \right)^p v(t) \dt\right)^\frac{q}{p}. \label{YippeeKiYay}
        \end{align}
        Since $A_1\le A_3$, it follows from \eqref{thm:HanickaToLambda:suffAII} that
        \begin{equation}\label{thm:HanickaToLambda:suffCIeq10}
                \sum_{k\in\K}\left(\int_{x_k}^L h(s)\ds \right)^q\int_{x_{k-1}}^{x_k} w(t) \dt\lesssim A_3^{q}\left(\int_0^L\left(\int_0^L\frac{U(s)h(s)}{U(s)+U(t)}\ds \right)^p v(t) \dt\right)^\frac{q}{p}.
        \end{equation}
        Hence, starting with \eqref{thm:HanickaToLambda:equivalentexpression} and using \eqref{thm:HanickaToLambda:disantidisLHS}, \eqref{YippeeKiYay} and \eqref{thm:HanickaToLambda:suffCIeq10}, we obtain $C\lesssim A_3$ in case (iii).
       
        \emph{Case} (iv). Assume that $0<q<1$, $0<q<p$ and $A_4<\infty$.
        Denote
        \begin{equation}\label{thm:HanickaToLambda:defxialt}
                \xi(t)=\int_0^L \min\left\{U^\frac{q}{1-q}(s), U^\frac{q}{1-q}(t) \right\} W^\frac{q}{1-q}(s)w(s)U^{-\frac{q}{1-q}}(s) \ds,\quad 
t\in(0,L).
        \end{equation}
        By Lemma~\ref{lemma:changeofvariable}, we have
        \begin{equation}\label{thm:HanickaToLambda:defxi}
                \xi(t)\approx W^\frac1{1-q}(t) + U^\frac{q}{1-q}(t)\int_t^LW^\frac{q}{1-q}(s)w(s)U^{-\frac{q}{1-q}}(s)\ds \quad\text{for every } t\in(0,L).
        \end{equation}
        Thanks to \eqref{thm:HanickaToLambda:suffCIeq8}, the H\"{o}lder inequality with exponents $\frac{p}{q}$ and $\frac{r}{q}$, and \eqref{thm:HanickaToLambda:disantidisRHS}, we obtain
        \begin{align}
                &\sum_{k\in\K^+}\int_{x_{k-1}}^{x_k}\left(\int_t^{x_k} h(s)\ds \right)^q w(t) \dt \notag\\
                &\quad\le\left(\sum_{k\in\K^+}\left(\int_{x_{k-1}}^{x_k}h(t)\varphi^\frac1{p}(t) \dt\right)^p\right)^\frac{q}{p}\left(\sum_{k\in\K^+}\left(\int_{x_{k-1}}^{x_k}W^\frac{q}{1-q}(t)w(t)\varphi^{-\frac{q}{p(1-q)}}(t) \dt\right)^\frac{r(1-q)}{q}\right)^\frac{q}{r} \notag\\
                &\quad\approx\left(\int_0^L\left(\int_0^L\frac{U(s)h(s)}{U(s)+U(t)}\ds \right)^p v(t) \dt\right)^\frac{q}{p}\left(\sum_{k\in\K^+}\left(\int_{x_{k-1}}^{x_k}W^\frac{q}{1-q}(t)w(t)\varphi^{-\frac{q}{p(1-q)}}(t) \dt\right)^\frac{r(1-q)}{q}\right)^\frac{q}{r}. \label{thm:HanickaToLambda:suffDIeq10}
        \end{align}
        Furthermore, Lemma~\ref{lem:disantidisint} with $\widetilde h(t)=\frac{U(t)^r}{\varphi^\frac{r}{p}(t)}$, $\widetilde\varrho = U^r$, $\widetilde p=\frac{r(1-q)}{q}$ and $\widetilde f=W^\frac{q}{1-q}w$ gives
        \begin{equation}\label{thm:HanickaToLambda:suffDIeq11}
                \left(\sum_{k\in\K^+}\left(\int_{x_{k-1}}^{x_k}W^\frac{q}{1-q}(t)w(t)\varphi^{-\frac{q}{p(1-q)}}(t) \dt\right)^\frac{r(1-q)}{q}\right)^\frac{q}{r}\approx\left(\sum_{k\in\K_-^+}\frac{\xi^\frac{r(1-q)}{q}(x_k)}{\varphi^\frac{r}{p}(x_k)}\right)^\frac{q}{r}.
        \end{equation}
        Recall that \eqref{thm:HanickaToLambda:suffBIeq6} is valid for any $0<q<p<\infty$ and note that $W^\frac{r}{q}(t)\lesssim\xi^\frac{r(1-q)}{q}(t)$ for every $t\in(0,L)$.
        Therefore, \eqref{thm:HanickaToLambda:disantidisLHS}, \eqref{thm:HanickaToLambda:suffDIeq10}, \eqref{thm:HanickaToLambda:suffDIeq11} and \eqref{thm:HanickaToLambda:suffBIeq6} yield
        \begin{equation}\label{thm:HanickaToLambda:suffDIeq12}
                \left(\int_0^L\left(\int_t^L h(s)\ds \right)^qw(t) \dt\right)^\frac1{q}\lesssim\left(\sum_{k\in\K_-^+}\frac{\xi^\frac{r(1-q)}{q}(x_k)}{\varphi^\frac{r}{p}(x_k)}\right)^\frac1{r}\left(\int_0^L\left(\int_0^L\frac{U(s)h(s)}{U(s)+U(t)}\ds \right)^p v(t) \dt\right)^\frac1{p}.
        \end{equation}
        Combining \eqref{thm:HanickaToLambda:suffBIeq3} and Lemma~\ref{lem:disantidisint} with $\widetilde h=U^r\varphi^{-\frac{r}{p}}$, $\widetilde\varrho=U^r$, $\widetilde p=\frac{r(1-q)}{q}$, $\widetilde f=W^\frac{q}{1-q}w$, we obtain       
        \begin{equation}\label{thm:HanickaToLambda:suffDIeq13}
                \left(\sum_{k\in\K_-^+}\frac{\xi^\frac{r(1-q)}{q}(x_k)}{\varphi^\frac{r}{p}(x_k)}\right)^\frac1{r}\approx A_4.
        \end{equation}
        Hence, the desired upper bound $C\lesssim A_4$ follows from \eqref{thm:HanickaToLambda:suffDIeq12}, \eqref{thm:HanickaToLambda:suffDIeq13} 
and \eqref{thm:HanickaToLambda:equivalentexpression}.
        \\
							
        \emph{Lower bounds.} Now we shall turn our attention to proving the lower bounds. Suppose that $C<\infty$. Fix an~arbitrary $t\in(0,L)$ and choose any function $f\in\M_\mu(X)$ such that
        \begin{equation*}
                f^* = \frac{\chi_{[0,t)}} {\left(V(t)+U^p(t)\int_t^Lv(s)U^{-p}(s)\ds \right)^\frac1{p}}.
        \end{equation*}
        Such a function indeed exists, see \cite[Corollary 7.8, p.~86]{BS}.
        Observe that $\|f\|_{\Gpuv}= 1$, and so
        \begin{equation*}
                \frac{W^\frac1{q}(t)}{\left(V(t)+U^p(t)\int_t^Lv(s)U^{-p}(s)\ds \right)^\frac1{p}} \le C
        \end{equation*}
        because the left-hand side is equal to $\|f\|_{\Lambda^{q}(w)}$.
        Since $t$ was arbitrary, we get the estimate $A_1 \le C$ by taking the supremum over $t\in(0,L)$.
        Notice that no additional assumptions on $p$ or $q$ were needed. Hence, not only does this complete case (i), but it also shows that $A_1\le 
C$ in all cases (i)--(iv). This is a~common feature of inequalities of this type.       
       
        Let us continue with the other cases.
        Thanks to \eqref{thm:HanickaToLambda:equivalentexpression}, \eqref{thm:HanickaToLambda:disantidisLHS} and \eqref{thm:HanickaToLambda:disantidisRHS}, we have
        \begin{equation}\label{thm:HanickaToLambda:necesseq}
                \left(\sum_{k\in\K^+}\int_{x_{k-1}}^{x_k}\left(\int_t^{x_k} h(s)\ds \right)^q w(t) \dt\right)^\frac1{q}\lesssim C \left(\sum_{k\in\K^+}\left(\int_{x_{k-1}}^{x_k}\varphi^\frac1{p}(t)h(t) \dt\right)^p\right)^\frac1{p}
        \end{equation}
        for every $h\in\Mpl(0, L)$.
       
        Exploiting the saturation of the Hardy inequality (see \cite[Lemma 5.4]{OK:90} and \cite[Theorem 3.3]{SS:96}) and \eqref{thm:HanickaToLambda:necesseq}, by the same argument as in \citep[pages 340--344]{GP:03} we 
obtain the following estimates:
				\begin{itemize}
        \item If $1\le q<\infty$, $p > q$, then
        \begin{equation} \label{thm:HanickaToLambda:necessB} 
                \left(\sum_{k\in\K^+}\left(\sup_{t\in(x_{k-1},x_k]}\varphi^{-\frac{q}{p}}(t)\int_{x_{k-1}}^tw(s)\ds \right)^\frac{r}{q}\right)^\frac1{r}\lesssim C;
        \end{equation}       
        \item If $0<q<1$, $p\le q$, then
        \begin{equation}\label{thm:HanickaToLambda:necessC}       
                \sup_{k\in\K^+}\left(\int_{x_{k-1}}^{x_k}\left(\int_{x_{k-1}}^tw(s)\ds \right)^\frac{q}{1-q}w(t)\varphi^{-\frac{q}{p(1-q)}}(t) \dt\right)^\frac{1-q}{q}\lesssim C;
        \end{equation}
        \item If $0<q<1$, $p > q$, then
        \begin{equation}\label{thm:HanickaToLambda:necessD}
                \left(\sum_{k\in\K^+}\left(\int_{x_{k-1}}^{x_k}\left(\int_{x_{k-1}}^tw(s)\ds \right)^\frac{q}{1-q}w(t)\varphi^{-\frac{q}{p(1-q)}}(t) \dt\right)^\frac{(1-q)r}{q}\right)^\frac1{r}\lesssim C.
        \end{equation}
				\end{itemize}
       
        \emph{Case} (ii). Assume that $1\le q < p< \infty$.
        We have
        \begin{align*}
                A_2 & \approx\left(\sum_{k\in\K^+}\sup_{t\in(x_{k-1},x_k]}\frac{W^\frac{r}{q}(t)}{\varphi^\frac{r}{p}(t)}\right)^\frac1{r}\\
                &\lesssim\left(\sum_{k=K_-+2}^{K^+}\frac{W^\frac{r}{q}(x_{k-1})}{\varphi^\frac{r}{p}(x_{k-1})}\right)^\frac1{r} + \left(\sum_{k\in\K^+}\sup_{t\in(x_{k-1},x_k]}\frac{\left(\int_{x_{k-1}}^tw(s)\ds \right)^\frac{r}{q}}{\varphi^\frac{r}{p}(t)}\right)^\frac1{r}\\
                &=\left(\sum_{k\in\K}\frac{\left(\sum_{l=K_-+1}^k\int_{x_{l-1}}^{x_l}w(s)\ds \right)^\frac{r}{q}}{\varphi^\frac{r}{p}(x_k)}\right)^\frac1{r} + \left(\sum_{k\in\K^+}\left(\sup_{t\in(x_{k-1},x_k]}\frac{\int_{x_{k-1}}^tw(s)\ds }{\varphi^\frac{q}{p}(t)}\right)^\frac{r}{q}\right)^\frac1{r}\\
                &\approx\left(\sum_{k\in\K}\left(\frac{\int_{x_{k-1}}^{x_k}w(s)\ds }{\varphi^\frac{q}{p}(x_k)}\right)^\frac{r}{q}\right)^\frac1{r} 
+ \left(\sum_{k\in\K^+}\left(\sup_{t\in(x_{k-1},x_k]}\frac{\int_{x_{k-1}}^tw(s)\ds }{\varphi^\frac{q}{p}(t)}\right)^\frac{r}{q}\right)^\frac1{r}\\
                &\lesssim C.
        \end{align*}
        Here, the first and last step are based on \eqref{thm:HanickaToLambda:suffBIeq4} and \eqref{thm:HanickaToLambda:necessB}, respectively. The fourth step follows from \cite[Lemma~1.3.4]{EGO:18} combined with \eqref{prel:CSgeomincr}.
       
        \emph{Case} (iii). Let $0<p\le q<1$.
        Owing to \eqref{thm:HanickaToLambda:suffCIeq9}, we have
        \begin{align}
                A_3 & \approx \sup_{k\in\K^+}\left(\int_{x_{k-1}}^{x_k}W^\frac{q}{1-q}(t)w(t)\varphi^{-\frac{q}{p(1-q)}}(t) \dt\right)^\frac{1-q}{q} \notag \\
                &\approx\sup_{\substack{K_-+2\le k \le K^+\\k\in\Z}}W(x_{k-1})\left(\int_{x_{k-1}}^{x_k}w(t)\varphi^{-\frac{q}{p(1-q)}}(t) \dt\right)^\frac{1-q}{q} \notag \\
                &\quad+ \sup_{k\in\K^+}\left(\int_{x_{k-1}}^{x_k}\left(\int_{x_{k-1}}^tw(s)\ds \right)^\frac{q}{1-q}w(t)\varphi^{-\frac{q}{p(1-q)}}(t) \dt\right)^\frac{1-q}{q}. \label{thm:HanickaToLambda:necessCeq1}
        \end{align}
        As for the first term, one may use exactly the same argument as in \cite[page~343]{GP:03} to obtain
        \begin{align}
                &\sup_{\substack{K_-+2\le k \le K^+\\k\in\Z}}W(x_{k-1})\left(\int_{x_{k-1}}^{x_k}w(t)\varphi^{-\frac{q}{p(1-q)}}(t) \dt\right)^\frac{1-q}{q} \notag\\
                &\qquad\lesssim\sup_{k\in\K}\varphi^{-\frac{1}{p}}(x_k)W^\frac1{q}(x_k) \notag \\
                &\qquad\quad+ \sup_{\substack{K_-+2\le k \le K^+\\k\in\Z}}\left(\int_{x_{k-1}}^{x_k}\left(\int_{x_{k-1}}^tw(s)\ds \right)^\frac{q}{1-q}w(t)\varphi^{-\frac{q}{p(1-q)}}(t) \dt\right)^\frac{1-q}{q}. \label{thm:HanickaToLambda:necessCeq2}
        \end{align}
        Moreover,
        \begin{align}
                \sup_{k\in\K}\varphi^{-\frac{1}{p}}(x_k)W^\frac1{q}(x_k) & \approx\sup_{k\in\K}\varphi^{-\frac{1}{p}}(x_k)\left(\int_{x_{k-1}}^{x_k}w(t) \dt\right)^\frac1{q} \notag\\
                &\approx\sup_{k\in\K}\varphi^{-\frac{1}{p}}(x_k)\left(\int_{x_{k-1}}^{x_k}\left(\int_{x_{k-1}}^tw(s)\ds \right)^\frac{q}{1-q}w(t) \dt\right)^\frac{1-q}{q} \notag\\
                &\le\sup_{k\in\K}\left(\int_{x_{k-1}}^{x_k}\left(\int_{x_{k-1}}^tw(s)\ds \right)^\frac{q}{1-q}w(t)\varphi^{-\frac{q}{p(1-q)}}(t) \dt\right)^\frac{1-q}{q}, \label{thm:HanickaToLambda:necessCeq3}
        \end{align}
        where the first equivalence is valid thanks to \cite[Lemma~1.3.4]{EGO:18} again, and the second one follows from Lemma~\ref{lemma:changeofvariable}. Hence, combining \eqref{thm:HanickaToLambda:necessCeq1}, \eqref{thm:HanickaToLambda:necessCeq2} and \eqref{thm:HanickaToLambda:necessCeq3} with \eqref{thm:HanickaToLambda:necessC}, we get
        $A_3\lesssim C.$
       
        \emph{Case} (iv). Let $0<q<1$, $0<q<p$. Similarly to the previous 
case, it can be shown (cf.~\cite[pages 344--346]{GP:03}) that
        \begin{align}
                &\left(\sum_{k\in\K^+}\left(\int_{x_{k-1}}^{x_k}W^\frac{q}{1-q}(t)w(t)\varphi^{-\frac{q}{p(1-q)}}(t) \dt\right)^\frac{(1-q)r}{q}\right)^\frac1{r} \notag\\
                &\quad\approx\left(\sum_{k=K_-+2}^{K^+}W^r(x_{k-1})\left(\int_{x_{k-1}}^{x_k}w(t)\varphi^{-\frac{q}{p(1-q)}}(t) \dt\right)^\frac{(1-q)r}{q}\right)^\frac1{r} \notag\\
                &\quad\quad+\left(\sum_{k\in\K^+}\left(\int_{x_{k-1}}^{x_k}\left(\int_{x_{k-1}}^tw(s)\ds \right)^\frac{q}{1-q}w(t)\varphi^{-\frac{q}{p(1-q)}}(t) \dt\right)^\frac{(1-q)r}{q}\right)^\frac1{r} \notag\\
                &\quad\approx\left(\sum_{k\in\K^+}\left(\int_{x_{k-1}}^{x_k}\left(\int_{x_{k-1}}^tw(s)\ds \right)^\frac{q}{1-q}w(t)\varphi^{-\frac{q}{p(1-q)}}(t) \dt\right)^\frac{(1-q)r}{q}\right)^\frac1{r}. \label{thm:HanickaToLambda:necessDeq1}
        \end{align}
        Hence, the desired inequality $A_4\lesssim C$ follows from \eqref{thm:HanickaToLambda:necessDeq1}, \eqref{thm:HanickaToLambda:suffDIeq11}, 
\eqref{thm:HanickaToLambda:suffDIeq13} and \eqref{thm:HanickaToLambda:necessD}.
\end{proof}

\begin{remark}\label{rem:alternativeexpressionofA4}
        Keeping the setting of Theorem~\ref{thm:HanickaToLambda}, we make 
the following remark. If
        \begin{equation}\label{rem:HanickaToLambda:assumW}
                \int_{t}^LW^\frac{q}{1-q}(s)w(s)U^{-\frac{q}{1-q}}(s)\ds  
< \infty \quad \text{ for every } t\in(0,L),
        \end{equation}
        then $A_4$ (and so also $C$) is equivalent to $A_5$, where
        \begin{equation*}
                A_5 = \left(\int_0^L\frac{\left(W^\frac{1}{1-q}(t)+U^\frac{q}{1-q}(t)\int_t^LW^\frac{q}{1-q}(s)w(s)U^{-\frac{q}{1-q}}(s)\ds \right)^{\frac{p(1-q)}{p-q}-1}W^\frac{q}{1-q}(t)w(t)}{\left(V(t)+U^p(t)\int_t^Lv(s)U^{-p}(s)\ds \right)^\frac{q}{p-q}} \dt\right)^\frac{p-q}{pq}.
        \end{equation*}
        We shall prove this assertion. Note that assumption \eqref{rem:HanickaToLambda:assumW} implies that the function $\xi$ defined by \eqref{thm:HanickaToLambda:defxialt} is finite and, moreover, $\xi\in Q_{U^\frac{q}{1-q}}(0,L)$. Therefore, there is a~covering sequence $\{\widetilde{x}_k\}_{k\in\widetilde{\K}^+_-}\in CS(\xi, U^\frac{q}{1-q}, b)$ for each parameter $b>1$. We find $b$ sufficiently large so that the assumptions of the theorems that we are to use are satisfied. The sufficient size of $b$ depends only on $p$ and $q$. Combining Theorem~\ref{thm:disantidis} applied to $\widetilde h=\xi$, $\widetilde\varrho=U^\frac{q}{1-q}$, $\d\widetilde\nu(t)=W^\frac{q}{1-q}(t)w(t)U^{-\frac{q}{1-q}}(t)\dt$, $\widetilde\alpha=\widetilde\beta=0$, $\widetilde p=\frac{r(1-q)}{q}$, $\widetilde f=U^r\varphi^{-\frac{r}{p}}$ with \eqref{thm:HanickaToLambda:defxi}, we obtain
        \begin{equation}\label{rem:HanickaToLambda:eq1}
                \sum_{k\in\widetilde{\K}_-^+}\frac{\xi^\frac{r(1-q)}{q}(\widetilde{x}_k)}{\varphi^\frac{r}{p}(\widetilde{x}_k)}\approx A_5^{r}.
        \end{equation}
        Furthermore, since both functions $\xi^\frac{r(1-q)}{q}$ and $U^r\varphi^{-\frac{r}{p}}$ are $U^r$-quasiconcave on $(0,L)$, we have
        \begin{equation}\label{rem:HanickaToLambda:eq2}
                \sum_{k\in\K_-^+}\frac{\xi^\frac{r(1-q)}{q}(x_k)}{\varphi^\frac{r}{p}(x_k)}\approx\sum_{k\in\widetilde{\K}_-^+}\frac{\xi^\frac{r(1-q)}{q}(\widetilde{x}_k)}{\varphi^\frac{r}{p}(\widetilde{x}_k)}
        \end{equation}
        thanks to \cite[Lemma~4.2.9]{EGO:18}, where $\{x_k\}_{k\in\K^+_-}$ is the covering sequence from the proof of Theorem~\ref{thm:HanickaToLambda}. Hence the desired equivalence follows from \eqref{thm:HanickaToLambda:suffDIeq13} combined with \eqref{rem:HanickaToLambda:eq1} and \eqref{rem:HanickaToLambda:eq2}.
       
        Without the additional assumption, $A_4$ and $A_5$ need not, however, be equivalent. In order to see this, note that $\xi\equiv\infty$ if \eqref{rem:HanickaToLambda:assumW} is violated. If this is the case, then 
$A_4=\infty$ provided that \eqref{thm:HanickaToLambda:assumV} is true, but $A_5=0$ provided that $\frac{p(1-q)}{p-q}-1<0$ (which is the case when $0<q<1<p<\infty$). It appears that this peculiar detail was overlooked in \cite[Theorem~4.2]{GP:03}.
\end{remark}

The final theorem, which generalizes \citep[Theorem~1.8]{GP:06} by allowing degenerated weights, deals with a~variant of the main result in the setting $p=\infty$. It provides an equivalent estimate on the optimal constant $C$ in the inequality
\begin{equation*}
                \left( \int_0^L (\f(t))^q w(t)\dt \right)^\frac 1q \le C \esssup_{t\in(0,L)} \left(\frac 1{U(t)} \int_0^t \f(s) u(s) \ds\right) v(t),
\end{equation*}
which is expressed by \eqref{thm:HanickaToLambdaweakdefC} below. We note that the representation \eqref{thm:HanickaToLambdaweakreprephi} below is always possible (see~Remark~\ref{rem:HanickaToLambdaweak}).
\begin{theorem}\label{thm:HanickaToLambdaweak}
        Let $q\in(0,\infty)$. Let $v,w$ be weights on $(0,L)$ and $u$ an a.e.~positive weight on $(0,L)$. Set
        \begin{equation}\label{thm:HanickaToLambdaweakdefC}
                C = \sup_{\|f\|_{\Gpuvweak}\le1}\|f\|_{\Lambda^{q}(w)},
        \end{equation}
        and
        \begin{equation}\label{thm:HanickaToLambdaweak:defphi}
                \varphi(t)=\esssup_{\tau\in(0,t)}U(\tau)\esssup_{s\in(\tau,L)}\frac{v(s)}{U(s)},\ t\in(0,L).
        \end{equation}
        Let $B_1, B_2\in(0,\infty)$, $\gamma, \delta\in[0,\infty)$ and $\nu$ a nonnegative Borel measure on $(0,L)$ such that
        \begin{equation}\label{thm:HanickaToLambdaweakreprephi}
                B_1\varphi(t)\le\gamma + \delta U(t) + \int_{(0,L)}\min\{U(t), U(s)\}\dnu(s)\le B_2\varphi(t)\quad\text{for every $t\in(0,L)$}.
        \end{equation}
        \begin{enumerate}[(i)]
                \item If $1\le q<\infty$, then $C\approx A_6$, where
                \begin{align*}
                        &A_6=\left(\lim_{t\to0^+}\frac{U(t)}{\varphi(t)}\right)\left(\sup_{t\in(0,L)}\frac{W^\frac1{q}(t)}{U(t)}\right)+ \lim_{t\to L^-}\frac1{\varphi(t)} W^\frac1{q}(L)\\
                        &\quad+\left(\int_0^LU^q(t)\left(\sup_{\tau\in(t,L)}\frac{W(\tau)}{U^q(\tau)}\right)\varphi^{-(q+2)}(t)u(t)\left(\gamma+\int_{(0,t]}U(s)\dnu(s)\right)\left(\delta+\int_{[t,L)}\dnu(s)\right) \dt\right)^\frac1{q}.
                \end{align*}
                \item If $0<q<1$, then $C\approx A_7$, where
                \begin{align*}
                        &A_7= \left(\lim_{t\to0^+}\frac{U(t)}{\varphi(t)}\right) \left(\int_0^LW^\frac{q}{1-q}(t)w(t)U^{-\frac{q}{1-q}}(t) \dt\right)^\frac{1-q}{q} +  \lim_{t\to L^-}\frac1{\varphi(t)} W^\frac1{q}(L)\\
                        &\quad+\left(\int_0^L\xi(t)\varphi^{-(q+2)}(t)u(t)\left(\gamma+\int_{(0,t]}U(s)\dnu(s)\right)\left(\delta+\int_{[t,L)}\dnu(s)\right) \dt\right)^\frac1{q},
                \end{align*}
                where
                \begin{equation}\label{thm:HanickaToLambdaweak:defxi}
                        \xi(t)=\left(\int_0^L W^\frac{q}{1-q}(s)w(s)U^{-\frac{q}{1-q}}(s)\min\{U^\frac{q}{1-q}(t), U^\frac{q}{1-q}(s)\}\ds \right)^{1-q},\ t\in(0,L).
                \end{equation}
                The equivalence constants depend only on the parameter $q$ and the constants $B_1$ and $B_2$. In particular, they are independent of the weights $u$, $v$ and $w$.
        \end{enumerate}
\end{theorem}

\begin{proof}
        We start off with a few useful observations. Note that
        \begin{equation*}
                \esssup_{t\in(0,L)}\frac1{U(t)}\int_0^t \f(s) u(s)\ds\,v(t)=\esssup_{t\in(0,L)}\frac1{U(t)}\int_0^t \f(s) u(s)\ds\,\varphi(t)
        \end{equation*}
        for every $f\in\M_\mu(X)$ (see~\citep[Lemma~1.5]{GP:06}). Coupling this with an argument similar to that leading to \eqref{thm:HanickaToLambda:equivalentexpression}, we obtain
        \begin{equation}\label{thm:HanickaToLambdaweak:equivalentexpression}
                C \approx \sup_{h\in\Mpl(0,L)}\frac{\left(\int_0^L\left(\int_t^L h(s)\ds \right)^q w(t) \dt\right)^\frac1{q}}{\sup_{t\in(0,L)}\varphi(t)\int_0^L\frac{U(s)h(s)}{U(t)+U(s)}\ds}.
        \end{equation}
        Furthermore, without loss of generality, we may assume that
        \begin{equation}\label{thm:HanickaToLambdaweak:assumphi}
                \varphi(t)<\infty\quad\text{for every $t\in(0,L)$},
        \end{equation}
        for, if this is not the case, then $\varphi$ is actually identically equal to $\infty$ on $(0,L)$, and so it follows that $C=A_6=A_7=0$ (Convention~\ref{conv}(\ref{conv:indefinite}) is used here once again).
       
        By interchanging the order of the suprema, we get the identity
        \begin{equation*}
                \varphi(t)=\esssup_{s\in(0,L)} v(s)\min\left\{1,\frac{U(t)}{U(s)}\right\}\quad\text{for every $t\in(0,L)$}.
        \end{equation*}
        Hence $\varphi\in Q_U(0,L)$.
       
        Since all key ideas were already presented in the proof of Theorem~\ref{thm:HanickaToLambda}, we will only outline the proof instead of going into all detail.
       
        Since $\varphi\in Q_U(0,L)$, there is a covering sequence $\{x_k\}_{k\in\K^+_-}\in CS(\varphi,U, a)$ for each $a>1$. We fix such a sequence for $a>1$ sufficiently large so that the assumptions of the theorems that we are to use are satisfied. An explicit estimate on $a$ may be obtained by careful examination of each step of the proof. What is important is 
that it depends only on the parameter $q$ and on the constants $B_1$ and $B_2$.
       
        By Lemma~\ref{lem:discantisup} applied to $\widetilde p=1$, $\widetilde\varphi=\varphi$, $\widetilde\varrho=U$ and $\widetilde f=Uh$, we have
        \begin{equation*}
                \sup_{t\in(0,L)}\varphi(t)\int_0^L\frac{U(s)h(s)}{U(t)+U(s)}\ds \approx\sup_{k\in\K^+_-}\varphi(x_k)\int_0^L\frac{U(s)h(s)}{U(x_k)+U(s)}\ds \approx\sup_{k\in\K^+}\int_{x_{k-1}}^{x_k}h(s)\varphi(s)\ds .
        \end{equation*}
        By discretizing the right-hand side of \eqref{thm:HanickaToLambdaweak:equivalentexpression} as in \eqref{thm:HanickaToLambda:disantidisLHS} and arguing as in \citep[the proof of Theorem~1.8]{GP:06}, one can show 
that
        \begin{align}
                C^q\approx\sum_{k\in\K^+}\sup_{t\in(x_{k-1},x_k]}\frac{W(t)}{\varphi^q(t)}\quad&\text{if $1\le q<\infty$}\label{thm:HanickaToLambdaweak:optimalniqvetsialpha}\\
                \intertext{and}
                C^q\approx\sum_{k\in\K^+_-}\frac{\xi(x_k)}{\varphi^q(x_k)}\quad&\text{if $0<q<1$}.\label{thm:HanickaToLambdaweak:optimalniqmensialpha}
        \end{align}
        In order to obtain the desired results, we need to anti-discretize the right-hand sides of \eqref{thm:HanickaToLambdaweak:optimalniqvetsialpha} and \eqref{thm:HanickaToLambdaweak:optimalniqmensialpha}.
       
        Thanks to \citep[Theorem~2.4.4]{EGO:18}, using the representation 
\eqref{thm:HanickaToLambdaweakreprephi} of $\varphi$, we have, for every $t\in(0,L)$,
        \begin{equation}\label{thm:HanickaToLambdaweakreprephireciproq}
                \begin{aligned}
                        &\frac{U^q(t)}{\varphi^q(t)}\approx\left(\lim_{t\to0^+}\frac{U(s)}{\varphi(s)}\right)^q + \left(\lim_{s\to L^-}\frac1{\varphi(s)}\right)^pU^q(t)\\
                        &\quad+\int_0^L\min\{U^q(t), U^q(s)\}\varphi^{-(q+2)}(s)u(s)\left(\gamma+\int_{(0,s]}U(\tau)\dnu(\tau)\right)\left(\delta+\int_{[s,L)}\dnu(\tau)\right)\ds.
                \end{aligned}
        \end{equation}
        The equivalence constants in \eqref{thm:HanickaToLambdaweakreprephireciproq} depend only on $q$, $A_1$ and $A_2$.
       
        If $1\leq q$, the equivalence $C\approx A_6$ follows from \eqref{thm:HanickaToLambdaweak:optimalniqvetsialpha}, \eqref{thm:HanickaToLambdaweakreprephireciproq} and Lemma~\ref{lem:disantidissup} 
applied to $\widetilde h=U^{q}\varphi^{-q}$, $\widetilde\varrho=U^q$, 
$\widetilde q=1$ and $\widetilde f=W$.
       
        Assume now that $0<q<1$. Observe that
        \begin{equation}\label{thm:HanickaToLambdaweakalternativexi}
                \begin{aligned}
                        \xi(t)&\approx W(t) + U^q(t)\left(\int_t^LW^\frac{q}{1-q}(s)w(s)U^{-\frac{q}{1-q}}(s)\ds \right)^{1-q}\\
                        &\approx U^{q}(t)\left(\int_0^L \frac{W^\frac{q}{1-q}(s)w(s)}{U^{\frac{q}{1-q}}(t)+U^{\frac{q}{1-q}}(s)}\ds \right)^{1-q}
                \end{aligned}
        \end{equation}
        for every $t\in(0,L)$ thanks to Lemma~\ref{lemma:changeofvariable}. Furthermore, by the same reasoning it is easy to see that
        \begin{equation}\label{thm:HanickaToLambdaweak:limit}
                \left(\int_0^LW^\frac{q}{1-q}(s)w(s)\ds \right)^{1-q}\approx W(L).
        \end{equation}
        The equivalence $C\approx A_7$ follows from \eqref{thm:HanickaToLambdaweak:optimalniqmensialpha}, \eqref{thm:HanickaToLambdaweakalternativexi}, \eqref{thm:HanickaToLambdaweak:limit} and Lemma~\ref{lem:disantidisint} applied to $\widetilde h=U^{q}\varphi^{-q}$, $\widetilde\varrho=U^q$, $\widetilde q=1-q$, $\widetilde f=W^\frac{q}{1-q}w$ and $\widetilde\alpha,\widetilde\beta,\widetilde\nu$ given by \eqref{thm:HanickaToLambdaweakreprephireciproq}.
\end{proof}

\begin{remark}\label{rem:HanickaToLambdaweak}
        Since the function $\varphi$ defined by \eqref{thm:HanickaToLambdaweak:defphi} is in $Q_U(0,L)$, there is always a nonnegative Borel measure $\nu$ on $(0,L)$ that represents $\varphi$ as in \eqref{thm:HanickaToLambdaweakreprephi} with $B_1=1$, $B_2=4$, $\gamma=\lim_{t\to0^+}\varphi(t)$ and $\delta=\lim_{t\to L^-}\frac{\varphi(t)}{U(t)}$ (recall \eqref{prel:repreofQq}).
       
        Whereas the fundamental function of $\Gpuv$ has an integral form if $p\in(0,\infty)$, the fundamental function of $\Gpuvweak$ is given by a~supremum. This is the reason why the statement of Theorem~\ref{thm:HanickaToLambdaweak} is somewhat more implicit than that of Theorem~\ref{thm:HanickaToLambda}.
        Nevertheless, under some extra assumptions, which are often satisfied in applications, we can actually represent $\varphi$ in the form of \eqref{thm:HanickaToLambdaweakreprephi} quite explicitly.
       
        For example, if $v\in Q_U(0,L)$, $v$ is differentiable on $(0,L)$, and $\frac{v'}{u}$ is nonincreasing and locally absolutely continuous on $(0,L)$, then $\varphi=v$ and \eqref{thm:HanickaToLambdaweakreprephi} 
holds with $B_1=B_2=1$, $\gamma=\lim_{s\to0^+}v(s)$, $\delta=\lim_{s\to L^-}\frac{v'(s)}{u(s)}$, and $\dnu(t)=\left(-\frac{v'}{u}\right)'(t)\dt$. This can be proved by integrating by parts upon observing that $\lim_{s\to0^+}U(s)\frac{v'(s)}{u(s)}=0$.
        Moreover, under these extra assumptions, we can also take $B_1=1$, $B_2=2$, $\gamma=\lim_{s\to0^+}v(s)$, $\delta=\lim_{s\to L^-}\frac{v(s)}{U(s)}$, and $\dnu(t)=\left(-\frac{v'}{u}\right)'(t)\dt$ thanks to the fact that $\lim_{s\to L^-}\frac{v(s)}{U(s)}\ge\lim_{s\to L^-}\frac{v'(s)}{u(s)}$ and $\lim_{s\to L^-}\frac{v(s)}{U(s)}\leq \frac{v(t)}{U(t)}$ for every $t\in(0,L)$.
\end{remark}

\begin{remark}       
        In the setting of Theorem~\ref{thm:HanickaToLambdaweak}, if $0<q<1$, there is an alternative, equivalent expression for $A_7$ under the extra assumption
        \begin{equation}\label{thm:HanickaToLambdaweak:assumxi}
                \xi(t)<\infty\quad\text{for every $t\in(0,L)$},
        \end{equation}
        where $\xi$ is defined by \eqref{thm:HanickaToLambdaweak:defxi} (cf.~Remark~\ref{rem:alternativeexpressionofA4}). Namely, $A_7\approx A_8$, where
        \begin{equation*}
                A_8=\left(\int_0^L\varphi(t)^{-q}\xi(t)^{-\frac{q}{1-q}}W^\frac{q}{1-q}(t)w(t) \dt\right)^\frac1{q}.
        \end{equation*}
        Indeed, the additional assumption implies that $\xi\in Q_{U^q}(0, 
L)$, and the desired equivalence then follows from \eqref{thm:HanickaToLambdaweak:optimalniqmensialpha} coupled with the fact that $C\approx A_7$, 
and Theorem~\ref{thm:disantidis} applied to $\widetilde h=\widetilde{\xi}^\frac{1}{1-q}$, $\widetilde\varrho=U^\frac{q}{1-q}$, $\widetilde\alpha=\widetilde\beta=0$, $\d\widetilde\nu(t)=W^\frac{q}{1-q}(t)w(t)U^{-\frac{q}{1-q}}(t)\dt$, $\widetilde p={1-q}$ and $\widetilde f=U^q\varphi^{-q}$ coupled with \citep[Lemma~4.2.9]{EGO:18}.
       
        The constants $A_7$ and $A_8$ need not, however, be equivalent if 
\eqref{thm:HanickaToLambdaweak:assumxi} is violated. In order to see this, suppose that \eqref{thm:HanickaToLambdaweak:assumphi} holds but \eqref{thm:HanickaToLambdaweak:assumxi} does not. Then $\xi\equiv\infty$ on $(0,L)$, and thus $A_7=\infty$ but $A_8=0$. This detail was probably overlooked in \citep[Theorem~1.8]{GP:06}.
\end{remark}

\begin{remark}\label{rem:sob}
We conclude this paper by outlining a possible application of our results. Let $\Omega\subseteq\mathbb R^n$, $n\ge2$, be a bounded Lipschitz domain, $m\in\N$, $m<n$, and $d\in(0,n-m)$. Let $\mu$ be a $d$-upper Ahlfors measure on $\overline\Omega$, that is, a finite Borel measure on $\overline\Omega$ such that
		\begin{equation*}
			\sup_{\substack{x\in\mathbb R^n\\r>0}}\frac{\mu\big(B(x,r)\cap \overline\Omega\big)}{r^d}<\infty,
		\end{equation*}
		where $B(x,r)$ is the open ball in $\mathbb R^n$ of radius $r$ centered 
at~$x$. Notable examples of such measures are $d$-dimensional Hausdorff measures on $d$-dimensional sets. If $X$ and $Y$ are rearrangement-invariant function spaces satisfying
		\begin{align*}
			\Bigg\|\int_{t^\frac{n}{d}}^1f^*(s)s^{-1+\frac{m}{n}}\ds\,\Bigg\|_{Y(0,1)}&\lesssim \|f\|_{X(0,1)}\\
			\intertext{and}
			\Bigg\|\,t^{-\frac{m}{n-d}}\int_0^\frac{n}{d}f^*(s)s^{-1+\frac{m}{n-d}}\ds\,\Bigg\|_{Y(0,1)}&\lesssim \|f\|_{X(0,1)}
		\end{align*}
		for every $f\in\Mpl(0,1)$, then \cite[Theorem~5.1]{CPS:20} ensures boundedness of a linear Sobolev trace operator
		\begin{equation}\label{rem:sob:tremb}
			\operatorname{T}\colon W^m X(\Omega)\to Y^{\langle\frac{n-d}{m}\rangle}(\Omega,\mu),
		\end{equation}
		where $W^m X(\Omega)$ is a Sobolev-type space of $m$-th order built upon $X(\Omega)$ and $Y^{\langle\frac{n-d}{m}\rangle}(\Omega,\mu)$ is the rearrangement-invariant function space whose norm is defined as
		\begin{equation*}
			\|u\|_{Y^{\langle\frac{n-d}{m}\rangle}(\Omega,\mu)}=\left\|\big(\big(g_u^\frac{n-d}{m}\big)^{**}\big)^\frac{m}{n-d}\right\|_{Y(0,1)},\quad u\in\M_\mu(\overline\Omega),
		\end{equation*}
		where $g_u(t)=u^*(\mu(\overline\Omega) t)$, $t\in(0,1)$.
		
		Since a large number of customary rearrangement-invariant function spaces are instances of Lorentz $\Lambda$-spaces, it is of interest to know how to apply this result of \cite{CPS:20} when $Y$ is a Lorentz $\Lambda$-space. Note that despite this assumption the resulting target space in \eqref{rem:sob:tremb} need not be equivalent to a~$\Lambda$-space (see \citep{hanicka}). 
		
		Supposing that $Y=\Lpv$, one might ask if the target space in \eqref{rem:sob:tremb} may be replaced by~$\Lambda^q(w)$. The answer is positive if $Y^{\langle\frac{n-d}{m}\rangle}(\Omega,\mu)$ embeds in $\Lambda^q(w)$, that is, if there is a~constant $C>0$ such that
		\begin{equation*}
			\left( \int_0^L (\f(t))^{q} w(t)\dt \right)^\frac 1{q} \leq C \left( \int_0^L \left( \frac 1{t} \int_0^t (\f(s))^\frac{n-d}{m}\ds \right)^{\frac{pm}{n-d}} v(t) \dt \right)^\frac 1{p}
		\end{equation*}
		holds for every $f\in\Mpl(0,L)$ with $L=\mu(\overline\Omega)$. This is where the main results of this paper come into play because, by a~standard rescaling argument, the optimal constant $C$ in the inequality above satisfies
		\begin{equation*}
			C^\frac{n-d}{m}=\sup_{\|f\|_{\Gamma^{\widetilde p}(v)\le1}}\|f\|_{\Lambda^{\widetilde q}(w)},
		\end{equation*}
		where $\widetilde p=\frac{pm}{n-d}$ and $\widetilde q=\frac{qm}{n-d}$. Notably, the absence of the ``non-degeneracy'' restrictions is crucial because $L<\infty$.
		
		Furthermore, the results obtained in this paper could also be used to improve some of the compactness results for the Sobolev trace operator \eqref{rem:sob:tremb} obtained in \cite[Theorem~5.3]{CM:20}.
\end{remark}

\paragraph{Acknowledgments}
The authors would like to thank the anonymous referees for their remarks and suggestions, which have led to improvements of the final version of the article.

\end{document}